\documentclass{article}

\usepackage{color}

\usepackage{amsmath}
\usepackage{latexsym}
\usepackage{amssymb}
\usepackage{amsthm}
\usepackage{amscd}
\usepackage{amssymb,amsmath}
\usepackage{MnSymbol}
\usepackage{mathrsfs}
\usepackage{dsfont}

\newtheorem{thm}{Theorem}
\newtheorem{lem}[thm]{Lemma}
\newtheorem{cor}[thm]{Corollary}
\newtheorem{prop}[thm]{Proposition}
\newtheorem{fac}[thm]{Fact}
\newtheorem{cla}[thm]{Claim}

\theoremstyle{definition}
\newtheorem{rem}[thm]{Remark}
\newtheorem{nota}[thm]{Notation}
\newtheorem{defi}[thm]{Definition}
\newtheorem{defi+nota}[thm]{Definition and Notation}

\newtheorem*{conv}{Convention}
\newtheorem*{fore}{Warning}

\title{Partially critical 2-structures}

\author{
Houmem Belkhechine\thanks{Carthage University, 
Bizerte Preparatory Engineering Institute, BP 64, 7021 Bizerte, Tunisia; {\tt houmem@gmail.com}.}
\and Imed Boudabbous\thanks{University of Sfax, Department of Mathematics, 
B.P. 802, 3038 Sfax, Tunisia; {\tt imed.boudabbous@gmail.com}.}
\and Pierre Ille\thanks{Aix Marseille Univ, CNRS, Centrale Marseille, I2M UMR 7373, 
13453 Marseille, France; {\tt pierre.ille@univ-amu.fr}.}
}

\begin{document}

\maketitle

\begin{abstract}
A 2-structure $\sigma$ consists of a vertex set $V(\sigma)$ and of an equivalence relation $\equiv_\sigma$ defined on $(V(\sigma)\times V(\sigma))\setminus\{(v,v):v\in V(\sigma)\}$. 
Given a 2-structure $\sigma$, a subset $M$ of $V(\sigma)$ is a module of $\sigma$ if for $x,y\in M$ and $v\in V(\sigma)\setminus M$, 
$(x,v)\equiv_{\sigma}(y,v)$ and $(v,x)\equiv_{\sigma}(v,y)$. 
For instance, 
$\emptyset$, $V(\sigma)$ and $\{v\}$, for $v\in V(\sigma)$, are modules of $\sigma$ called trivial modules of $\sigma$. 
A 2-structure $\sigma$ is prime if $v(\sigma)\geq 3$ and all the modules of 
$\sigma$ are trivial. 
A prime 2-structure $\sigma$ is critical if for each $v\in V(\sigma)$, $\sigma-v$ is not prime. 
A prime 2-structure $\sigma$ is partially critical if there exists $X\subsetneq V(\sigma)$ such that $\sigma[X]$ is prime, and for each $v\in V(\sigma)\setminus X$, $\sigma-v$ is not prime. 
We characterize finite or infinite partially critical 2-structures. 
\end{abstract}

\medskip

\noindent {\bf Mathematics Subject Classifications (2010):} 05C75, 05C63, 06A05.

\medskip

\noindent {\bf Key words:} 2-structure, module, prime, critical, partially critical. 

\section{Introduction}\label{s_intro}

The 2-structures were introduced by Ehrenfeucht et al.~\cite{EHR99}. 
They are well adapted generalizations of binary combinatorial structures like graphs, tournaments,... within the framework of modular decomposition. 
We consider finite or infinite 2-structures. 

A module (or a clan~\cite{EHR99}) of a 2-structure is a subset such that each vertex outside is linked in the same way to all the vertices inside. 
A 2-structure is prime if all its modules are trivial. 
In a finite and prime 2-structure, we can remove one or two vertices in order to obtain a prime 2-substructure. 
This result if false for infinite and prime 2-structures. 
In fact, there exist infinite and prime 2-structures that become non-prime 
after removing any finitely many vertices. 
In the sequel, such prime 2-structures are called finitely critical. 
A vertex $v$ of a prime 2-structure is critical (in terms of primality) when the 2-substructure obtained by removing $v$ is not prime. 
Now, a prime 2-structure is critical if all its vertices are critical. 
The finite and critical 2-structures were characterized independently by 
Bonizzoni~\cite{B94}, and Schmerl and Trotter~\cite{ST93}. 
The problem of the characterization of infinite and critical 2-structures remains open. 
The central difficulty comes from the existence of finitely critical 2-structures. 
Nevertheless, Boudabbous and Ille ~\cite{BI07} succeeded in characterizing infinite and prime digraphs that are critical, but not finitely critical. 

A prime 2-structure is partially critical if every vertex outside a prime induced 2-substructure is critical. 
Finite and partially critical graphs were characterized by Breiner et al.~\cite{BDI08}. 
Finite and partially critical tournaments were characterized by Sayar~\cite{S11} who adapted the examination of partial criticality presented in \cite{BDI08} to tournaments.

Almost all finite and prime 2-structures are prime. 
Thus, it is impossible to characterize or to describe the finite and prime 2-structures of a given cardinality. 
Now, suppose that a finite and prime 2-structure admits a critical vertex. 
The withdrawal of this vertex creates a partial module, which imposes conditions 
on the 2-structure. 
When the 2-structure is critical, that is, when all its vertices are critical, we obtain so many conditions that it is possible to characterize the finite and critical 2-structures up to isomorphism (see~\cite{B94} and \cite{ST93}). 
For finite and partially critical 2-structures, we have less conditions, and we do not succeed in characterizing them up to isomorphism. 
Nevertheless, we can localize the created partial modules because of the prime induced 2-substructure, which leads us to a description by using an auxiliary graph. 

In this paper, we characterize finite or infinite partially critical 2-structures. 
For the finite case, we follow the same approach as that of \cite{BDI08}. 
We associate with the prime induced 2-substructure its outside graph (see Definition~\ref{outsideg}). 
For a finite and partially critical 2-structure, the components of its outside graph are critical and bipartite (see Theorem~\ref{thm_main_2}), that is, are half graphs (see Proposition~\ref{prop1_half_graph}). 
This result establishes an important structural link between partial criticality and (global) criticality via the outside graph. 
Furthermore, always in the finite case, if we add an odd number of vertices to the prime induced 2-substructure, we obtain a non-prime induced 2-substructure. 
This fact is false in the infinite case when we consider finitely critical 2-structures as particular partially critical 2-structures. 
Therefore, to study infinite and partially critical 2-structures, we suppose that the addition of 5 vertices to the prime induced 2-substructure gives a non-prime induced 2-substructure. 
Under this assumption, we can proceed by compactness. 
We obtain that the components of the outside graph are critical and $P_5$-free bipartite graphs. 
It turns out that the critical and $P_5$-free bipartite graphs are the half graphs defined from a discrete linear order (see Theorem~\ref{thm1_half_graph}). 

At present, we formalize our presentation. 
A {\em 2-structure}~\cite{EHR99} $\sigma$ consists of a finite or infinite {\em vertex set} $V(\sigma)$, and of an equivalence relation $\equiv_\sigma$ defined on $(V(\sigma)\times V(\sigma))\setminus\{(v,v):v\in V(\sigma)\}$. 
The cardinality of $V(\sigma)$ is denoted by $v(\sigma)$. 
The set of the equivalence classes of $\equiv_\sigma$ is denoted by 
$E(\sigma)$. 
Given a 2-structure $\sigma$, 
with each $W\subseteq V(\sigma)$ associate the {\em 2-substructure} $\sigma[W]$ of $\sigma$ induced by $W$ defined on $V(\sigma[W])=W$ such that 
$$(\equiv_{\sigma[W]})\ =\ (\equiv_{\sigma})_{\restriction (W\times W)\setminus\{(w,w):w\in W\}}.$$
Given $W\subseteq V(\sigma)$, 
$\sigma[V(\sigma)\setminus W]$ is denoted by $\sigma-W$, and by $\sigma-w$ when $W=\{w\}$. 

A graph $\Gamma=(V(\Gamma),E(\Gamma))$ is identified with the 2-structure 
$\sigma_\Gamma$ defined on 
$V(\sigma_\Gamma)=V(\Gamma)$ as follows. 
For $u,v,x,y\in V(\Gamma)$ such that $u\neq v$ and $x\neq y$, 
$(u,v)\equiv_{\sigma_\Gamma}(x,y)$ if 
$\{u,v\},\{x,y\}\in E(\Gamma)$ or $\{u,v\},\{x,y\}\not\in E(\Gamma)$. 
Similarly, 
a tournament $T=(V(T),A(T))$ is identified with the 2-structure $\sigma_T$ defined on $V(\sigma_T)=V(T)$ as follows. 
For $u,v,x,y\in V(T)$ such that $u\neq v$ and $x\neq y$, 
$(u,v)\equiv_{\sigma_T}(x,y)$ if $(u,v),(x,y)\in A(T)$ or $(u,v),(x,y)\not\in A(T)$.

\subsection{Prime 2-structures}\label{sub_prime}

We remind the important results on prime 2-structures.  

\begin{conv}
Let $\sigma$ be a 2-structure. 
For $X\subseteq V(\sigma)$, $\overline{X}$ denotes $V(\sigma)\setminus X$. 
\end{conv}

Let $\sigma$ be a 2-structure. 
A subset $M$ of $V(\sigma)$ is a {\em module} \cite{S92} of $\sigma$ if for any $x,y\in M$ and 
$v\in\overline{M}$, we have 
$$(x,v)\equiv_{\sigma}(y,v)\ \text{and}\ (v,x)\equiv_{\sigma}(v,y).$$
For instance, 
$\emptyset$, $V(\sigma)$ and $\{v\}$, for $v\in V(\sigma)$, are modules of $\sigma$ called {\em trivial} modules of $\sigma$. 
A 2-structure $\sigma$ is {\em prime} if $v(\sigma)\geq 3$ and all the modules of 
$\sigma$ are trivial. 
The main definitions follow.

\begin{defi}\label{defi_critical}
Given a prime 2-structure $\sigma$, 
a vertex $v$ of $\sigma$ is {\em critical} (in terms of primality) if $\sigma-v$ is not prime. 
More generally, a subset $W$ of $V(\sigma)$ is {\em critical} if $\sigma-W$ is not prime. 
A prime 2-structure is {\em critical} if all its vertices are critical. 

Let $\sigma$ be a prime 2-structure. 
Given $W\subseteq V(\sigma)$, $\sigma$ is {\em $W\!$-critical} if all the elements of $W$ are critical vertices of $\sigma$. 
Lastly, a prime 2-structure $\sigma$ is {\em partially critical} if 
there exists $X\subsetneq V(\sigma)$ such that $\sigma[X]$ is prime, and $\sigma$ 
is $\overline{X}\!$-critical. 
\end{defi}

\begin{nota}\label{nota_p}
Let $\sigma$ be a 2-structure. 
With $X\subsetneq V(\sigma)$ such that $\sigma[X]$ is prime, 
associate the following subsets of $\overline{X}$ 
\begin{itemize}
\item ${\rm Ext}_\sigma(X)$ is the set of $v\in\overline{X}$ such that $\sigma[X\cup\{v\}]$ is prime;
\item $\langle X\rangle_\sigma$ is the set of $v\in\overline{X}$ such that $X$ is a module of 
$\sigma[X\cup\{v\}]$;
\item given $\alpha\in X$, $X_\sigma(\alpha)$ is the set of $v\in\overline{X}$ such that 
$\{\alpha,v\}$ is a module of $\sigma[X\cup\{v\}]$. 
\end{itemize}
The set $\{{\rm Ext}_\sigma(X),\langle X\rangle_\sigma\}\cup\{X_\sigma(\alpha):\alpha\in X\}$ is denoted by 
$p_{(\sigma,\overline{X})}$. 
It is called the {\em outside partition}. 
\end{nota}

The next result (see \cite[Lemmas~6.3 and 6.4]{EHR99}) is basic in the study of primality.

\begin{lem}\label{lem_EHR}
Given a 2-structure $\sigma$, consider $X\subsetneq V(\sigma)$ such that $\sigma[X]$ is prime. 
The set $p_{(\sigma,\overline{X})}$ is a partition of $\overline{X}$. 
Moreover, the three assertions below hold 
\begin{enumerate}
\item for $v\in\langle X\rangle_\sigma$ and $w\in\overline{X}\setminus\langle X\rangle_\sigma$, 
if $\sigma[X\cup\{v,w\}]$ is not prime, then $X\cup\{w\}$ is a module of $\sigma[X\cup\{v,w\}]$;
\item given $\alpha\in X$, for $v\in X_\sigma(\alpha)$ and 
$w\in\overline{X}\setminus X_\sigma(\alpha)$, 
if $\sigma[X\cup\{v,w\}]$ is not prime, then $\{\alpha,v\}$ is a module of $\sigma[X\cup\{v,w\}]$;
\item for distinct $v,w\in{\rm Ext}_\sigma(X)$, if $\sigma[X\cup\{v,w\}]$ is not prime, then $\{v,w\}$ is a module of $\sigma[X\cup\{v,w\}]$.
\end{enumerate}
\end{lem}

The classic parity theorem \cite[Theorem~6.5]{EHR99} follows from Lemma~\ref{lem_EHR}. 

\begin{thm}\label{tEHR}
Given a 2-structure $\sigma$, consider $X\subsetneq V(\sigma)$ such that $\sigma[X]$ is prime and 
$|\overline{X}|\geq 2$. 
If $\sigma$ is prime, then there exist distinct $v,w\in\overline{X}$ such that 
$\sigma[X\cup\{v,w\}]$ is prime. 
\end{thm}

Theorem~\ref{tEHR} leads us to introduce the outside graph as follows. 
We need the next notation. 

\begin{nota}\label{nota_outsideg}
Given a 2-structure $\sigma$, consider $X\subsetneq V(\sigma)$ such that $\sigma[X]$ is prime. 
The set of the nonempty subsets $Y$ of $\overline{X}$, such that $\sigma[X\cup Y]$ is prime, is denoted by 
$\varepsilon_{(\sigma,\overline{X})}$. 
Hence ${\rm Ext}_\sigma(X)=\{v\in\overline{X}:\{v\}\in\varepsilon_{(\sigma,\overline{X})}\}$. 
Furthermore, suppose that $|\overline{X}|\geq 2$. 
By Theorem~\ref{tEHR}, $\varepsilon_{(\sigma,\overline{X})}$ contains an unordered pair. 
\end{nota}

\begin{defi}\label{outsideg}
Given a 2-structure $\sigma$, consider $X\subsetneq V(\sigma)$ such that $\sigma[X]$ is prime. 
The {\em outside graph} 
$\Gamma_{(\sigma,\overline{X})}$ is defined on 
$\overline{X}$ by 
$$E(\Gamma_{(\sigma,\overline{X})})=
\{Y\in\varepsilon_{(\sigma,\overline{X})}:|Y|=2\}.$$
By Theorem~\ref{tEHR}, $\Gamma_{(\sigma,\overline{X})}$ is nonempty when 
$|\overline{X}|\geq 2$. 
The outside graph is a common tool in the study of prime graphs~\cite{I97,IR14}. 
\end{defi}

By applying Theorem~\ref{tEHR} several times, we obtain the following result. 

\begin{cor}\label{cor_EHR}
Given a 2-structure $\sigma$, consider $X\subsetneq V(\sigma)$ such that $\sigma[X]$ is prime. 
Suppose that $\overline{X}$ is finite, with $|\overline{X}|\geq 2$. 
If $\sigma$ is prime, then there exist $v,w\in\overline{X}$ such that 
$\sigma-\{v,w\}$ is prime. 
\end{cor}

Schmerl and Trotter~\cite{ST93} characterized the finite and critical 2-structures (see Definition~\ref{defi_critical}). 
Using their characterization, they obtained the following improvement of Corollary~\ref{cor_EHR}, which is an important result on the finite and 
prime 2-structures. 

\begin{thm}\label{thm_ST}
Given a finite and prime 2-structure $\sigma$, if $v(\sigma)\geq 7$, then there exist distinct vertices $v$ and $w$ of $\sigma$ such that $\sigma-\{v,w\}$ is prime. 
\end{thm}

In the next theorem, Ille~\cite{I97} succeeded in localizing a non-critical unordered pair outside a prime 2-substructure. 
Initially, it was established for finite digraphs. 
The same proof holds for finite 2-structures. 

\begin{thm}\label{thm_pi}
Given a prime 2-structure $\sigma$, consider $X\subsetneq V(\sigma)$ such that 
$\sigma[X]$ is prime. 
If $\overline{X}$ is finite and $|\overline{X}|\geq 6$, then there exist distinct $v,w\in\overline{X}$ such that $\sigma-\{v,w\}$ is prime. 
\end{thm}

Sayar~\cite{S11} improved Theorem~\ref{thm_pi} for finite tournaments as follows. 

\begin{thm}\label{thm_mys}
Given a prime tournament $T$, consider $X\subsetneq V(T)$ such that 
$T[X]$ is prime. 
If $\overline{X}$ is finite and $|\overline{X}|\geq 4$, then there exist distinct $v,w\in\overline{X}$ such that $T-\{v,w\}$ is prime. 
\end{thm}

We extend Theorem~\ref{thm_mys} to particular 2-structures in 
Appendix~\ref{A_thm_pi} (see Theorem~\ref{thm_mys_ext}).

\begin{rem}\label{rem1_conditions_Ck}
Given a 2-structure $\sigma$, consider $X\subsetneq V(\sigma)$ such that $\sigma[X]$ is prime, and 
$$\text{$\overline{X}$ is finite.}$$ 
Suppose that $\sigma$ is $\overline{X}\!$-critical. 
For a contradiction, suppose that $|\overline{X}|$ is odd. 
By applying several times Theorem~\ref{tEHR} from $\sigma[X]$, we obtain 
a non-critical vertex $v$ of $\sigma$ such that $v\in\overline{X}$,  
which contradicts the fact that $\sigma$ is $\overline{X}\!$-critical. 
It follows that $|\overline{X}|$ is even. 

Now, consider $Y\subsetneq\overline{X}$ such that $\sigma[X\cup Y]$ is prime. 
Since $\sigma$ is $\overline{X}\!$-critical, $\sigma$ is $\overline{(X\cup Y)}$-critical as well. 
Therefore $|\overline{X\cup Y}|$ is even. 
Since $|\overline{X}|$ is even, $|Y|$ is even too. 
Consequently, for each $k\in\{1,\ldots,|\overline{X}|-1\}$ such that $k$ is odd, 
we have the following statement  
\begin{equation}\label{E1_thm_main_1}
\{Y\in\varepsilon_{(\sigma,\overline{X})}:|Y|=k\}=\emptyset.
\tag{{\rm Sk}}
\end{equation}
Clearly, ${\rm Ext}_\sigma(X)=\emptyset$ means that Statement (S1) holds. 

Lastly, consider $k\in\{1,\ldots,|\overline{X}|-1\}$ such that $k$ is odd. 
Suppose that Statement~(Sk) holds. 
It follows from Theorem~\ref{tEHR} that Statement~(Sm) holds for 
every odd integer $m\in\{1,\ldots,k-2\}$. 
\end{rem}

\subsection{Infinite and prime 2-structures}

Concerning infinite and prime 2-structures, Ille~\cite{I94,I05b} obtained the following two theorems. 
Initially, they were proved for digraphs. 
The same proofs hold for 2-structures. 

\begin{thm}\label{thm1_pi_infinite}
Given a prime 2-structure $\sigma$, consider $X\subsetneq V(\sigma)$ such that 
$\sigma[X]$ is prime. 
For each $x\in\overline{X}$, there exists $F\in\varepsilon_{(\sigma,\overline{X})}$ such that $F$ is finite and $x\in F$. 
\end{thm}

The next result follows from Theorem~\ref{thm1_pi_infinite}. 

\begin{cor}\label{cor1_pi_infinite}
Given a 2-structure $\sigma$, consider $X\subsetneq V(\sigma)$ such that 
$\sigma[X]$ is prime. 
The following two assertions are equivalent 
\begin{enumerate}
\item $\sigma$ is prime;
\item for each finite subset $F$ of $\overline{X}$, there exists $F'\in\varepsilon_{(\sigma,\overline{X})}$ such that $F'$ is finite and 
$F\subseteq F'$. 
\end{enumerate}
\end{cor}

The next compactness result follows from Corollary~\ref{cor1_pi_infinite}. 

\begin{thm}\label{thm2_pi_infinite}
Given an infinite 2-structure $\sigma$, the following two assertions are equivalent
\begin{enumerate}
\item $\sigma$ is prime;
\item for each finite subset $F$ of $V(\sigma)$, there exists a finite subset $F'$ of $V(\sigma)$ satisfying $F\subseteq F'$ and $\sigma[F']$ is prime. 
\end{enumerate}
\end{thm}

\begin{defi}\label{defi_finitely}
Given an infinite and prime 2-structure $\sigma$, $\sigma$ is {\em finitely critical} if 
$\sigma-F$ is not prime for every nonempty and finite subset $F$ of $V(\sigma)$. 
It follows from Theorem~\ref{tEHR} that a prime 2-structure $\sigma$ is finitely critical if and only if $\sigma-\{v,w\}$ is not prime for any $v,w\in V(\sigma)$. 
\end{defi}

Boudabbous and Ille~\cite{BI07} characterized the critical digraphs  
that are not finitely critical, that is, the infinite and prime digraphs 
$D$ satisfying 
\begin{itemize}
\item for each $v\in V(D)$, $D-v$ is not prime;
\item there exist (distinct) $v,w\in V(D)$ such that $D-\{v,w\}$ is prime. 
\end{itemize}

\subsection{Main results}\label{Main results}

We begin with a hereditary property of primality through the components of the outside graph, which constitutes the central result of the paper. 

\begin{thm}\label{thm_main_1}
Given a 2-structure $\sigma$, consider $X\subsetneq V(\sigma)$ such that $\sigma[X]$ is prime. 
Suppose that Statement (S3) holds. 
The following three assertions are equivalent 
\begin{enumerate}
\item $\sigma$ is prime;
\item for each component $C$ of $\Gamma_{(\sigma,\overline{X})}$, $\sigma[X\cup V(C)]$ is prime;
\item for each component $C$ of $\Gamma_{(\sigma,\overline{X})}$, $v(C)=2$ or $v(C)\geq 4$ and $C$ is prime. 
\end{enumerate}
\end{thm}

Theorem~\ref{thm_main_1} allows us to provide a simple and short proof of Theorem~\ref{thm_pi} 
(see Appendix~\ref{A_thm_pi}). 
Furthermore, Theorem~\ref{thm_main_1} is proved for finite graphs in \cite{IR14} 
(see \cite[Theorem 17]{IR14} and \cite[Corollary 18]{IR14}). 
We pursue with a hereditary property of partial criticality through the components of the outside graph. 

\begin{thm}\label{thm_main_2}
Given a 2-structure $\sigma$, consider $X\subsetneq V(\sigma)$ such that $\sigma[X]$ is prime. 
Suppose that Statement (S5) holds. 
The following three assertions are equivalent 
\begin{enumerate}
\item $\sigma$ is $\overline{X}\!$-critical;
\item for each component $C$ of $\Gamma_{(\sigma,\overline{X})}$, $\sigma[X\cup V(C)]$ is 
$V(C)$-critical;
\item for each component $C$ of $\Gamma_{(\sigma,\overline{X})}$, $v(C)=2$ or $v(C)\geq 4$ and $C$ is critical. 
\end{enumerate}
\end{thm}

Given a 2-structure $\sigma$, consider $X\subsetneq V(\sigma)$ such that 
$\sigma[X]$ is prime. 
Suppose that Statement (S5) holds. 
Suppose also that $\sigma$ is $\overline{X}\!$-critical. 
Consider a component $C$ of $\mathcal{C}(\Gamma_{(\sigma,\overline{X})})$ such that $v(C)\geq 4$. 
It follows from Theorem~\ref{thm_main_2} that $C$ is critical. 
Moreover, since Statement (S5) holds, $P_5\not\leq C$ (see Lemma~\ref{lem1_component}), where 
for $n\geq 2$, $P_n$ denotes the path on $n$ vertices. 
In Theorem~\ref{thm1_half_graph} below, we characterize the bipartite graphs 
$\Gamma$ such that 
$P_5\not\leq\Gamma$ and $\Gamma$ is critical. 
We need the following three definitions. 

\begin{defi}\label{defi_infinite_half_graph}
Given a bipartite graph $\Gamma$, with bipartition 
$\{X,Y\}$, $\Gamma$ is a {\em half graph} \cite{EH85} if there exist a linear order $L$ defined on $X$, and a bijection $\varphi$ from $X$ onto $Y$ such that 
\begin{equation}\label{E1_defi_infinite_half_graph}
E(\Gamma)=\{\{x,\varphi(x')\}:x\leq x'\hspace{-2mm}\mod L\}.
\end{equation}
\end{defi}

\begin{rem}\label{rem0_infinite_half_graph}
Given a bipartite graph $\Gamma$, with bipartition 
$\{X,Y\}$. 
Suppose that $\Gamma$ is a half graph. 
There exist a linear order $L$ defined on $X$, and a bijection $\varphi$ from $X$ onto $Y$ such that 
\eqref{E1_defi_infinite_half_graph} holds. 
Given $x,y\in X$, we obtain that 
$$\text{$x\leq y\hspace{-2mm}\mod L$ if and only if $N_\Gamma(x)\supseteq N_\Gamma(y)$}.$$ 
Therefore, the linear order $L$ is unique. 
\end{rem}

\begin{defi}\label{defi_discrete}
A linear order $L$ is {\em discrete} \cite{R82} if the following two conditions are satisfied
\begin{enumerate}
\item for every $v\in V(L)$, if $v$ is not the smallest element of $L$, then $v$ admits a predecessor;
\item for every $v\in V(L)$, if $v$ is not the largest element of $L$, then $v$ admits a successor.
\end{enumerate}
\end{defi}

\begin{defi}\label{defi_discrete_half}
A half graph is {\em discrete} if the linear order $L$ in Definition~\ref{defi_infinite_half_graph} is discrete. 
\end{defi}

\begin{thm}\label{thm1_half_graph}
Given a bipartite graph $\Gamma$, with $v(\Gamma)\geq 4$, the following assertions are equivalent 
\begin{enumerate}
\item $\Gamma$ is a discrete half graph;
\item $P_5\not\leq \Gamma$ and $\Gamma$ is critical. 
\end{enumerate}
\end{thm}

We establish Theorem~\ref{thm1_half_graph} in Section~\ref{section_Half_graphs}. 
The next result follows from Theorems~\ref{thm_main_1} and \ref{thm_main_2}, Proposition~\ref{prop1_half_graph}, and Lemma~\ref{lem1_component}. 

\begin{cor}\label{cor1_thm_main_2}
Given a 2-structure $\sigma$, consider $X\subsetneq V(\sigma)$ such that $\sigma[X]$ is prime. 
Suppose that $$\text{$\overline{X}$ is finite.}$$
The following two assertions are equivalent 
\begin{enumerate}
\item Statement (S5) holds, and $\sigma$ is prime;
\item $\sigma$ is $\overline{X}\!$-critical. 
\end{enumerate}
\end{cor}

\begin{rem}\label{rem2_conditions_Ck}
Consider the path $P_\mathbb{Z}=(\mathbb{Z},\{(p,q):|p-q|=1\})$. 
We show that $P_\mathbb{Z}$ is prime by using Theorem~\ref{thm2_pi_infinite}. 
Indeed, let $F$ be a finite and nonempty subset of $P_\mathbb{Z}$. 
There exist $p,q\in\mathbb{Z}$ such that $p\leq\min(F)$, $q\geq\max(F)$ and $q-p\geq 3$. 
Clearly, $F\subseteq\{p,\ldots,q\}$, and $P_\mathbb{Z}[\{p,\ldots,q\}]\simeq P_{q-p+1}$. 
Since $q-p+1\geq 4$, $P_{q-p+1}$ and hence $P_\mathbb{Z}[\{p,\ldots,q\}]$ are prime. 
By Theorem~\ref{thm2_pi_infinite}, $P_\mathbb{Z}$ is prime. 

For every $z\in\mathbb{Z}$, $P_\mathbb{Z}-z$ is disconnected, and hence 
$P_\mathbb{Z}-z$ is not prime. 
Consequently $P_\mathbb{Z}$ is critical. 
In fact, $P_\mathbb{Z}$ is finitely critical. 

Set $X=\{z\in\mathbb{Z}:z\leq 0\}$. 
By Theorem~\ref{thm2_pi_infinite}, $P_\mathbb{Z}[X]$ is prime. 
Since $P_\mathbb{Z}$ is critical, $P_\mathbb{Z}$ is $\overline{X}$-critical. 
For every $k>0$, $P_\mathbb{Z}[X\cup\{1,\ldots,k\}]$ is prime by 
Theorem~\ref{thm2_pi_infinite}. 
Consequently, for every $k>0$, Statement (Sk) does not hold. 
Moreover, $\{1,2\}$ is the only edge of $\Gamma_{(P_\mathbb{Z},\overline{X})}$. 
Hence, for every $z\geq 3$, $z$ is an isolated vertex of 
$\Gamma_{(P_\mathbb{Z},\overline{X})}$. 
It follows that Theorem~\ref{thm_main_1} does not hold when Statement (S3) is not satisfied. 
Similarly, Theorem~\ref{thm_main_2} does not hold when Statement (S5) is not satisfied. 
\end{rem}

Corollary~\ref{cor1_pi_infinite} and the fact  that Statement (S5)  is supposed to be satisfied in Theorem~\ref{thm_main_2} lead us to introduce the next definition. 
The next definition is a weakening of the partial criticality (see Theorem~\ref{cor1_thm_main_1}). 

\begin{defi}\label{defi_finitely}
Given a 2-structure $\sigma$, consider $X\subsetneq V(\sigma)$ such that 
$\sigma[X]$ is prime. 
We say that $\sigma$ is {\em finitely $\overline{X}$-critical} if for each finite subset $F$ of $\overline{X}$, there exists a finite subset $F'$ of 
$\overline{X}$ such that $F\subseteq F'$ and $\sigma[X\cup F']$ is 
$(F')$-critical. 
\end{defi}

The next result follows from Corollaries~\ref{cor1_pi_infinite} and \ref{cor1_thm_main_2}.

\begin{thm}\label{cor1_thm_main_1}
Given a 2-structure $\sigma$, consider $X\subsetneq V(\sigma)$ such that $\sigma[X]$ is prime. 
The following two assertions are equivalent 
\begin{enumerate}
\item Statement (S5) holds, and $\sigma$ is prime;
\item $\sigma$ is finitely $\overline{X}$-critical. 
\end{enumerate}
\end{thm}

Theorem~\ref{cor1_thm_main_1} is discussed in Remark~\ref{rem_thm_24}. 
Precisely, in Remark~\ref{rem_thm_24}, we provide a prime 2-structure showing that we do not have a compactness theorem with partial criticality. 

The last main result is an immediate consequence of Theorem~\ref{thm_main_2} and Claim~\ref{cla4_half_graph}. 

\begin{thm}\label{thm_main_4}
Given a 2-structure $\sigma$, consider $X\subsetneq V(\sigma)$ such that $\sigma[X]$ is prime. 
Suppose that Statement (S5) holds. 
Suppose also that $\sigma$ is $\overline{X}\!$-critical. 
For each $x\in\overline{X}$, there exists $y\in\overline{X}\setminus\{x\}$ such that 
$\sigma-\{x,y\}$ is $(\overline{X}\setminus\{x,y\})\!$-critical. 
\end{thm}

\begin{rem}\label{rem_thm_main_4}
Given a 2-structure $\sigma$, consider $X\subsetneq V(\sigma)$ such that $\sigma[X]$ is prime. 
Suppose that Statement (S5) holds. 
Suppose also that $\sigma$ is $\overline{X}\!$-critical. 
Lastly, suppose that $\overline{X}$ is infinite. 
Consider a finite and nonempty subset $F$ of $\overline{X}$. 
By applying several times Theorem~\ref{thm_main_4}, we obtain a finite 
subset $F'$ of $\overline{X}$ such that $F\subseteq F'$ and 
$\sigma-F'$ is $(\overline{X}\setminus F')$-critical. 
Furthermore, it follows from Theorem~\ref{tEHR} that $|F'|$ is even. 
\end{rem}

In Appendix~\ref{A_representation}, we describe simply partially critical 2-structures. 
A nice presentation of finite and partially critical tournaments is provided in~\cite{BBH15}. 

\begin{fore}
As mentioned at the beginning of Section~\ref{s_intro}, we adopt the same approach as that of \cite{BDI08} to examine finite and partially critical 2-structures. 
In what follows, we omit the proof of a result when it is closed to that provided in \cite{BDI08}. 
\end{fore}

\section{Preliminaries}

We use the following notation. 

\begin{nota}\label{not_link}
Let $\sigma$ be a 2-structure. 
For $W,W'\subseteq V(\sigma)$, with $W\cap W'=\emptyset$, 
$W\longleftrightarrow_\sigma W'$ 
signifies that $(v,v')\equiv_\sigma (w,w')$ and $(v',v)\equiv_\sigma (w',w)$ for any 
$v,w\in W$ and $v',w'\in W'$. 
Given $v\in V(\sigma)$ and $W\subseteq V(\sigma)\setminus\{v\}$, 
$\{v\}\longleftrightarrow_\sigma W$ is also denoted by 
$v\longleftrightarrow_\sigma W$.  
The negation is denoted by $v\not\longleftrightarrow_\sigma W$. 

Given distinct vertices $v$ and $w$ of $\sigma$, the equivalence class of $(v,w)$ is denoted by $(v,w)_\sigma$. 
If we consider $\sigma$ as the function from 
$(V(\sigma)\times V(\sigma))\setminus\{(v,v):v\in V(\sigma)\}$ to $E(\sigma)$, which maps $(v,w)$ to $(v,w)_\sigma$, then $\sigma$ becomes a 2-structure labeled by $E(\sigma)$. 
Given distinct vertices $v$ and $w$ of $\sigma$, set 
$$[v,w]_\sigma=((v,w)_\sigma,(w,v)_\sigma).$$
Given $W,W'\subseteq V(\sigma)$ such that 
$W\longleftrightarrow_\sigma W'$, $(W,W')_\sigma$ denotes 
the equivalence class of 
$(w,w')$, where $w\in W$ and $w'\in W'$. 
Furthermore, set $$[W,W']_\sigma=((W,W')_\sigma,(W',W)_\sigma).$$
Lastly, given $v\in V(\sigma)$ and $W\subseteq V(\sigma)\setminus\{v\}$ such that 
$v\longleftrightarrow_\sigma W$, $(\{v\},W)_\sigma$ is also denoted by 
$(v,W)_\sigma$, and 
$[\{v\},W]_\sigma$ is also denoted by $[v,W]_\sigma$. 
\end{nota}

Let $\sigma$ be a 2-structure. 
Using Notation~\ref{not_link}, a subset $M$ of $V(\sigma)$ is a module of $\sigma$ if and only if for each $v\in\overline{M}$, we have $v\longleftrightarrow_\sigma M$. 

To continue, we examine the isolated vertices of an outside graph. 
We utilize the following remark. 

\begin{rem}\label{rem_p}
Given a 2-structure $\sigma$, consider $X\subsetneq V(\sigma)$ such that $\sigma[X]$ is prime. 
Consider distinct $x,y\in\overline{X}$. 
If $x,y\in\langle X\rangle_\sigma$, then $X$ is a module of $\sigma[X\cup\{x,y\}]$. 
Given $\alpha\in X$, if $x,y\in X_\sigma(\alpha)$, then $\{\alpha,x,y\}$ is a module of 
$\sigma[X\cup\{x,y\}]$. 
Consequently, for each $B\in p_{(\sigma,\overline{X})}\setminus\{{\rm Ext}_\sigma(X)\}$, 
$\Gamma_{(\sigma,\overline{X})}[B]$ is empty. 
In other words, if ${\rm Ext}_\sigma(X)=\emptyset$, then $\Gamma_{(\sigma,\overline{X})}$ is multipartite with partition $p_{(\sigma,\overline{X})}$ (see Lemma~\ref{lem_EHR}). 
\end{rem}

The proof of the next lemma is analogous to that of \cite[Lemma~2.7]{BDI08}. 

\begin{lem}\label{lem1_modules}
Given a 2-structure $\sigma$, consider $X\subsetneq V(\sigma)$ such that $\sigma[X]$ is prime. 
\begin{enumerate}
\item If $M$ is a module of $\sigma$ such that $X\subseteq M$, then
the elements of $\overline{M}$ are isolated vertices of $\Gamma_{(\sigma,\overline{X})}$.
\item Given $\alpha\in X$, if $M$ is a module of $\sigma$ such that $M\cap X=\{\alpha\}$, then the elements of $M\setminus\{\alpha\}$ are isolated vertices of $\Gamma_{(\sigma,\overline{X})}$.
\end{enumerate}
\end{lem}

The next result is an immediate consequence of Lemma~\ref{lem1_modules}. 

\begin{cor}\label{cor1_modules}
Given a 2-structure $\sigma$, consider $X\subsetneq V(\sigma)$ such that $\sigma[X]$ is prime. 
If $\sigma$ admits a nontrivial module $M$ such
that $M\cap X\neq\emptyset$, then $\Gamma_{(\sigma,\overline{X})}$ possesses isolated vertices.
\end{cor}

Now, we study the modules of the outside graph. 
We need the following refinement of the outside partition (see Notation~\ref{nota_p}). 

\begin{nota}\label{nota_q}
Given a 2-structure $\sigma$, consider $X\subsetneq V(\sigma)$ such that $\sigma[X]$ is prime. 
We consider the following subsets of $\overline{X}$ 
\begin{itemize}
\item for $e,f\in E(\sigma)$, $\langle X\rangle_\sigma^{(e,f)}$ is the set of 
$v\in\langle X\rangle_\sigma$ such that $(v,\alpha)\in e$ and $(\alpha,v)\in f$, where $\alpha\in X$;
\item for $e,f\in E(\sigma)$ and $\alpha\in X$, $X_\sigma^{(e,f)}(\alpha)$ is the set of 
$v\in X_\sigma(\alpha)$ such that $(v,\alpha)\in e$ and $(\alpha,v)\in f$. 
\end{itemize}
The set $\{{\rm Ext}_\sigma(X)\}\cup\{\langle X\rangle_\sigma^{(e,f)}:e,f\in E(\sigma)\}\cup\{X_\sigma^{(e,f)}(\alpha):e,f\in E(\sigma),\alpha\in X\}$ is denoted by 
$q_{(\sigma,\overline{X})}$. 
\end{nota}

\begin{lem}\label{lem2_modules}
Given a 2-structure $\sigma$, consider $X\subsetneq V(\sigma)$ such that $\sigma[X]$ is prime. 
Suppose that Statement (S1) holds. 
Given $M\subseteq\overline{X}$, 
if $M$ is a module of $\sigma$, then 
$M$ is a module of $\Gamma_{(\sigma,\overline{X})}$, and there exist 
$B_p\in p_{(\sigma,\overline{X})}$ and $B_q\in q_{(\sigma,\overline{X})}$ such that 
$M\subseteq B_q\subseteq B_p$, and $M$ is a module of $\sigma[B_p]$. 
\end{lem}

\begin{proof}
Consider a module $M$ of $\sigma$ such that $M\cap X=\emptyset$. 
Let $x\in M$. 
Denote by $B_q$ the unique block of $q_{(\sigma,\overline{X})}$ containing $x$. 
Consider $y\in M\setminus\{x\}$. 
Since $M$ is a module of $\sigma$ such that $M\cap X=\emptyset$, we have 
$\alpha\longleftrightarrow_\sigma \{x,y\}$ for every $\alpha\in X$. 
It follows that $y\in B_q$. 
Consequently $M\subseteq B_q$. 
Denote by $B_p$ the unique block of $p_{(\sigma,\overline{X})}$ containing $B_q$. 
We obtain $$M\subseteq B_q\subseteq B_p.$$
Since $M$ is a module of $\sigma$, $M$ is a module of 
$\sigma[B_p]$. 

Lastly, we prove that $M$ is a module of $\Gamma_{(\sigma,\overline{X})}$. 
Let $v\in\overline{X}\setminus M$. 
Recall that ${\rm Ext}_\sigma(X)=\emptyset$ because Statement (S1) holds. 
If $v\in B_p$, then it follows from Remark~\ref{rem_p} that 
$\{y,v\}\not\in E(\Gamma_{(\sigma,\overline{X})})$ for every $y\in M$. 
Hence suppose that $v\in\overline{X}\setminus B_p$. 
Since ${\rm Ext}_\sigma(X)=\emptyset$, we distinguish the following two cases. 
\begin{itemize}
\item Suppose that $B_p=\langle X\rangle_\sigma$. 
Let $\alpha\in X$. 
Recall that $x\in M$. 

First, suppose that $x\longleftrightarrow_\sigma\{\alpha,v\}$. 
Let $y\in M$. 
Since $M$ is a module of $\sigma$, we obtain $y\longleftrightarrow_\sigma\{\alpha,v\}$. 
Since $y\longleftrightarrow_\sigma X$, we obtain $y\longleftrightarrow_\sigma X\cup\{v\}$. 
Hence $X\cup\{v\}$ is a module of $\sigma[X\cup\{y,v\}]$. 
It follows that 
$\{y,v\}\not\in E(\Gamma_{(\sigma,\overline{X})})$ for every $y\in M$. 

Second, suppose that $x\not\longleftrightarrow_\sigma\{\alpha,v\}$. 
Let $y\in M$. 
Since $M$ is a module of $\sigma$, we obtain $y\not\longleftrightarrow_\sigma\{\alpha,v\}$. 
Hence $X\cup\{v\}$ is not a module of $\sigma[X\cup\{y,v\}]$. 
It follows from the first assertion of Lemma~\ref{lem_EHR} that 
$\{y,v\}\in E(\Gamma_{(\sigma,\overline{X})})$ for every $y\in M$. 
\item Suppose that $B_p=X_\sigma(\alpha)$, where $\alpha\in X$. 
Recall that $x\in M$. 

First, suppose that $v\longleftrightarrow_\sigma\{\alpha,x\}$. 
Let $y\in M$. 
Since $M$ is a module of $\sigma$, we obtain $v\longleftrightarrow_\sigma\{\alpha,y\}$. 
Since $\{\alpha,y\}$ is a module of $\sigma[X\cup\{y\}]$, 
$\{\alpha,y\}$ is a module of $\sigma[X\cup\{y,v\}]$. 
It follows that 
$\{y,v\}\not\in E(\Gamma_{(\sigma,\overline{X})})$ for every $y\in M$. 

Second, suppose that $v\not\longleftrightarrow_\sigma\{\alpha,x\}$. 
Let $y\in M$. 
Since $M$ is a module of $\sigma$, we obtain $v\not\longleftrightarrow_\sigma\{\alpha,y\}$. 
Thus 
$\{\alpha,y\}$ is not a module of $\sigma[X\cup\{y,v\}]$. 
It follows from the second assertion of Lemma~\ref{lem_EHR} that 
$\{y,v\}\in E(\Gamma_{(\sigma,\overline{X})})$ for every $y\in M$. \qedhere
\end{itemize}
\end{proof}

The opposite direction in Lemma~\ref{lem2_modules} is false. 
Nevertheless, it is true for (finite) graphs (see the second assertion of \cite[Lemma~2.6]{BDI08}). 
Moreover, the opposite direction in Lemma~\ref{lem2_modules} is true if we 
require that 
Statement~(S3) holds (see Corollary~\ref{cor1_oppo_modules} below). 

\section{The first results}\label{S_first_results}

The proof of the next fact is analogous to that of \cite[Lemma~4.3]{BDI08}. 

\begin{fac}\label{fac1_first_results}
Given a 2-structure $\sigma$, consider $X\subsetneq V(\sigma)$ such that $\sigma[X]$ is prime. 
Suppose that Statement (S3) holds. 
Given distinct elements $x,y,z$ of $\overline{X}$, 
if $\{x,y\},\{x,z\}\in E(\Gamma_{(\sigma,\overline{X})})$, then $\{y,z\}$ is a module 
of $\sigma[X\cup\{x,y,z\}]$, and hence there exists 
$B_q\in q_{(\sigma,\overline{X})}$ such that $y,z\in B_q$. 
\end{fac}

The proof of the next fact is analogous to that of \cite[Lemma~4.4]{BDI08}.

\begin{fac}\label{fac2_first_results}
Given a 2-structure $\sigma$, consider $X\subsetneq V(\sigma)$ such that $\sigma[X]$ is prime. 
Suppose that Statement (S3) holds. 
Given $B_p,D_p\in p_{(\sigma,\overline{X})}$, consider $x\in B_p$ and 
$y,z\in D_p$ such that $\{x,y\}\in E(\Gamma_{(\sigma,\overline{X})})$ and 
$\{x,z\}\not\in E(\Gamma_{(\sigma,\overline{X})})$. 
\begin{enumerate}
\item If $D_p=\langle X\rangle_\sigma$, then $X\cup\{x,y\}$ is a module of $\sigma[X\cup\{x,y,z\}]$.
\item If $D_p=X_\sigma(\alpha)$, where $\alpha\in X$, then $\{\alpha,z\}$ is 
a module of $\sigma[X\cup\{x,y,z\}]$. 
\end{enumerate}
\end{fac}

The next result follows from Fact~\ref{fac1_first_results}. 

\begin{cor}\label{cor1_oppo_modules}
Given a 2-structure $\sigma$, consider $X\subsetneq V(\sigma)$ such that $\sigma[X]$ is prime. 
Suppose that Statement (S3) holds. 
Consider $M\subseteq\overline{X}$ such that there exist 
$B_p\in p_{(\sigma,\overline{X})}$ and $B_q\in q_{(\sigma,\overline{X})}$ with 
$M\subseteq B_q\subseteq B_p$. 
Suppose that $M$ is a module of $\sigma[B_p]$. 
If $M$ is a module of $\Gamma_{(\sigma,\overline{X})}$, then 
$M$ is a module of $\sigma$. 
\end{cor}

\begin{proof}
Consider $x,y\in M$ and $v\in\overline{M}$. 
It suffices to verify that 
\begin{equation}\label{E1_cor1_oppo_modules}
v\longleftrightarrow_\sigma\{x,y\}.
\end{equation}
Since $M$ is a module of $\sigma[B_p]$, \eqref{E1_cor1_oppo_modules} holds when 
$v\in B_p\setminus M$. 
Furthermore, since $x$ and $y$ belong to the same block of $q_{(\sigma,\overline{X})}$, 
\eqref{E1_cor1_oppo_modules} holds when 
$v\in X$. 

Now, suppose that $v\in\overline{X\cup B_p}$. 
Since $M$ is a module of $\Gamma_{(\sigma,\overline{X})}$, we have 
\begin{align}\label{E2_cor1_oppo_modules}
&\{x,v\},\{y,v\}\in E(\Gamma_{(\sigma,\overline{X})})\nonumber\\ 
&\text{or}\\
&\{x,v\},\{y,v\}\not\in E(\Gamma_{(\sigma,\overline{X})}).  \nonumber
\end{align}

Suppose that $\{x,v\},\{y,v\}\in E(\Gamma_{(\sigma,\overline{X})})$. 
By Fact~\ref{fac1_first_results}, $\{x,y\}$ is a module of $\sigma[X\cup\{x,y,v\}]$, so 
$v\longleftrightarrow_\sigma\{x,y\}$. 

Lastly, 
suppose that $\{x,v\},\{y,v\}\not\in E(\Gamma_{(\sigma,\overline{X})})$. 
Since Statement (S3) holds, Statement (S1) holds by Remark~\ref{rem1_conditions_Ck}. 
Hence ${\rm Ext}_\sigma(X)=\emptyset$, and we distinguish the following two cases. 
\begin{itemize}
\item Suppose that $B_p=\langle X\rangle_\sigma$. 
Since $\{x,v\},\{y,v\}\not\in E(\Gamma_{(\sigma,\overline{X})})$, 
it follows from the first assertion of Lemma~\ref{lem_EHR} that $X\cup\{v\}$ is a module of 
$\sigma[X\cup\{x,v\}]$ and $\sigma[X\cup\{y,v\}]$. 
Given $\alpha\in X$, we obtain $x\longleftrightarrow_\sigma\{\alpha,v\}$ and 
$y\longleftrightarrow_\sigma\{\alpha,v\}$. 
Since $x,y\in B_q$ and $B_q\subseteq\langle X\rangle_\sigma$, 
$\alpha\longleftrightarrow_\sigma\{x,y\}$. 
It follows that $v\longleftrightarrow_\sigma\{x,y\}$. 

Consequently, \eqref{E1_cor1_oppo_modules} holds when $v\in\overline{X\cup B_p}$ and 
$B_p=\langle X\rangle_\sigma$. 

\item Suppose that $B_p=X_\sigma(\alpha)$, where $\alpha\in X$. 
Since $\{x,v\},\{y,v\}\not\in E(\Gamma_{(\sigma,\overline{X})})$, 
it follows from the second assertion of Lemma~\ref{lem_EHR} that $\{\alpha,x\}$ is a module of 
$\sigma[X\cup\{x,v\}]$, and $\{\alpha,y\}$ is a module of $\sigma[X\cup\{y,v\}]$. 
Therefore 
$v\longleftrightarrow_\sigma\{\alpha,x\}$ and $v\longleftrightarrow_\sigma\{\alpha,y\}$. 
It follows that $v\longleftrightarrow_\sigma\{x,y\}$. 

Consequently, \eqref{E1_cor1_oppo_modules} holds when $v\in\overline{X\cup B_p}$ and 
$B_p=X_\sigma(\alpha)$. \qedhere
\end{itemize}
\end{proof}

The next two results follow from Fact~\ref{fac2_first_results}.

\begin{cor}\label{cor0_first_results}
Given a 2-structure $\sigma$, consider $X\subsetneq V(\sigma)$ such that $\sigma[X]$ is prime. 
Suppose that Statement (S3) holds. 
Let $B_q\in q_{(\sigma,\overline{X})}$. 
For each $v\in\overline{X}\setminus B_q$, 
$\{x\in B_q:\{x,v\}\in E(\Gamma_{(\sigma,\overline{X})})\}$ and 
$\{x\in B_q:\{x,v\}\not\in E(\Gamma_{(\sigma,\overline{X})})\}$ are modules of 
$\sigma[B_q]$. 
Precisely, if 
$\{x\in B_q:\{x,v\}\in E(\Gamma_{(\sigma,\overline{X})})\}\neq\emptyset$ and 
$\{x\in B_q:\{x,v\}\not\in E(\Gamma_{(\sigma,\overline{X})})\}\neq\emptyset$, 
then the following two assertions hold. 
\begin{enumerate}
\item If $B_q=\langle X\rangle_\sigma^{(e,f)}$, where $e,f\in E(\sigma)$, then 
$$[\{x\in B_q:\{x,v\}\not\in E(\Gamma_{(\sigma,\overline{X})})\},
\{x\in B_q:\{x,v\}\in E(\Gamma_{(\sigma,\overline{X})})\}]_\sigma=(e,f).$$ 
\item If $B_q=X_\sigma^{(e,f)}(\alpha)$, where $\alpha\in X$ and $e,f\in E(\sigma)$, then 
$$[\{x\in B_q:\{x,v\}\not\in E(\Gamma_{(\sigma,\overline{X})})\},
\{x\in B_q:\{x,v\}\in E(\Gamma_{(\sigma,\overline{X})})\}]_\sigma=(f,e).$$ 
\end{enumerate}
\end{cor}

\begin{proof}
Let $v\in\overline{X}\setminus B_q$. 
Suppose that $\{x\in B_q:\{x,v\}\in E(\Gamma_{(\sigma,\overline{X})})\}\neq\emptyset$ and 
$\{x\in B_q:\{x,v\}\not\in E(\Gamma_{(\sigma,\overline{X})})\}\neq\emptyset$. 
Consider $x^+,z^-\in B_q$ such that 
$\{x^+,v\}\in E(\Gamma_{(\sigma,\overline{X})})$ and 
$\{z^-,v\}\not\in E(\Gamma_{(\sigma,\overline{X})})$. 
We distinguish the following two cases. 
\begin{enumerate}
\item Suppose that $B_q=\langle X\rangle_\sigma^{(e,f)}$, where $e,f\in E(\sigma)$. 
By the first assertion of Fact~\ref{fac2_first_results} applied to $x^+,z^-,v$, $X\cup\{x^+,v\}$ is a module of $\sigma[X\cup\{x^+,z^-,v\}]$. 
Since $z^-\in\langle X\rangle_\sigma^{(e,f)}$, we obtain $[z^-,x^+]_\sigma=(e,f)$. 
\item Suppose that $B_q=X_\sigma^{(e,f)}(\alpha)$, where $\alpha\in X$ and $e,f\in E(\sigma)$. 
By the second assertion of Fact~\ref{fac2_first_results} applied to $x^+,z^-,v$, $\{\alpha,z^-\}$ is a module of $\sigma[X\cup\{x^+,z^-,v\}]$. 
Hence $[z^-,x^+]_\sigma=[\alpha,x^+]_\sigma$. 
Since $x^+\in X_\sigma^{(e,f)}(\alpha)$, we obtain $[\alpha,x^+]_\sigma=(f,e)$, so 
$[z^-,x^+]_\sigma=(f,e)$. \qedhere
\end{enumerate}
\end{proof}

The proof of the next corollary is analogous to that of \cite[Corollary~4.5]{BDI08}. 
It follows from Lemma~\ref{lem_EHR} and Fact~\ref{fac2_first_results}. 

\begin{cor}\label{cor1_first_results}
Given a 2-structure $\sigma$, consider $X\subsetneq V(\sigma)$ such that $\sigma[X]$ is prime. 
Suppose that Statement (S3) holds. 
If $\sigma$ is prime, then $\Gamma_{(\sigma,\overline{X})}$ has no isolated vertices.
\end{cor}

We examine the blocks of the partitions $p_{(\sigma,\overline{X})}$ and $q_{(\sigma,\overline{X})}$ in the next three lemmas. 

\begin{lem}\label{l1_p_and_q}
Given a 2-structure $\sigma$, consider $X\subsetneq V(\sigma)$ such that $\sigma[X]$ is prime. 
Suppose that Statement (S3) holds. 
Consider $e,f\in E(\sigma)$, and $\alpha\in X$. 
If $\Gamma_{(\sigma,\overline{X})}$ does not have isolated vertices, then the following two assertions hold 
\begin{enumerate}
\item if $\langle X\rangle_\sigma^{(e,f)}\neq\emptyset$, then 
$\langle X\rangle_\sigma^{(e',f')}=\emptyset$ for any $e',f'\in E(\sigma)$ such that $\{e',f'\}\neq\{e,f\}$;
\item if $X_\sigma^{(e,f)}(\alpha)\neq\emptyset$, then 
$X_\sigma^{(e',f')}(\alpha)=\emptyset$ for any $e',f'\in E(\sigma)$ such that $\{e',f'\}\neq\{e,f\}$. 
\end{enumerate}
\end{lem}

\begin{proof}
Consider $e,f,e',f'\in E(\sigma)$. 
For the first assertion, suppose that there exist $x\in\langle X\rangle_\sigma^{(e,f)}$ and $x'\in\langle X\rangle_\sigma^{(e',f')}$. 
We have to prove that 
\begin{equation}\label{E1_p_and_q}
\{e,f\}=\{e',f'\}.
\end{equation} 
Since $x,x'\in\langle X\rangle_\sigma$, we have $\{x,x'\}\not\in E(\Gamma_{(\sigma,\overline{X})})$ by Remark~\ref{rem_p}. 
Furthermore, since $\Gamma_{(\sigma,\overline{X})}$ does not have isolated vertices, there exist $y,y'\in\overline{X}\setminus\{x,x'\}$ such that 
$\{x,y\},\{x',y'\}\in E(\Gamma_{(\sigma,\overline{X})})$. 
Suppose that $y=y'$. 
We obtain $\{y,x\},\{y,x'\}\in E(\Gamma_{(\sigma,\overline{X})})$. 
It follows from Fact~\ref{fac1_first_results} that $(e,f)=(e',f')$, so \eqref{E1_p_and_q} holds. 
We obtain the same conclusion when $\{x,y'\}\in E(\Gamma_{(\sigma,\overline{X})})$ or 
$\{x',y\}\in E(\Gamma_{(\sigma,\overline{X})})$. 
Thus, suppose that $y\neq y'$, and 
$\{x,y'\},\{x',y\}\not\in E(\Gamma_{(\sigma,\overline{X})})$. 
It follows from the first assertion of Fact~\ref{fac2_first_results} applied to $x,x',y'$ that 
$X\cup\{x',y'\}$ is a module of $\sigma[X\cup\{x,x',y'\}]$. 
Since $x\in\langle X\rangle_\sigma^{(e,f)}$, we obtain $(x,x')\in e$ and $(x',x)\in f$. 
Similarly, it follows from the first assertion of Fact~\ref{fac2_first_results} applied to $x,x',y$ that 
$(x',x)\in e'$ and $(x,x')\in f'$. 
Therefore $e=f'$ and $e'=f$. 
Consequently \eqref{E1_p_and_q} holds. 

For the second assertion, suppose that there exist $x\in X_\sigma^{(e,f)}(\alpha)$ and 
$x'\in X_\sigma^{(e',f')}(\alpha)$, where $\alpha\in X$. 
We have to prove that \eqref{E1_p_and_q} holds. 
Since $x,x'\in X_\sigma(\alpha)$, we have $\{x,x'\}\not\in E(\Gamma_{(\sigma,\overline{X})})$ by Remark~\ref{rem_p}. 
Furthermore, since $\Gamma_{(\sigma,\overline{X})}$ does not have isolated vertices, there exist $y,y'\in\overline{X}\setminus\{x,x'\}$ such that 
$\{x,y\},\{x',y'\}\in E(\Gamma_{(\sigma,\overline{X})})$. 
Suppose that $y=y'$. 
We obtain $\{y,x\},\{y,x'\}\in E(\Gamma_{(\sigma,\overline{X})})$. 
By Fact~\ref{fac1_first_results}, $(e,f)=(e',f')$, so \eqref{E1_p_and_q} holds. 
We obtain the same conclusion when $\{x,y'\}\in E(\Gamma_{(\sigma,\overline{X})})$ or 
$\{x',y\}\in E(\Gamma_{(\sigma,\overline{X})})$. 
Now, suppose that $y\neq y'$, and 
$\{x,y'\},\{x',y\}\not\in E(\Gamma_{(\sigma,\overline{X})})$. 
It follows from the second assertion of Fact~\ref{fac2_first_results} applied to $x,x',y'$ that 
$\{\alpha,x\}$ is a module of $\sigma[X\cup\{x,x',y'\}]$. 
Hence $(x',\alpha)\equiv_\sigma(x',x)$ and $(\alpha,x')\equiv_\sigma(x,x')$. 
Since $x'\in X_\sigma^{(e',f')}(\alpha)$, we obtain $(x',x)\in e'$ and $(x,x')\in f'$. 
Similarly, it follows from the second assertion of Fact~\ref{fac2_first_results} applied to $x,x',y$ that 
$(x,x')\in e$ and $(x',x)\in f$. 
Thus $e=f'$ and $e'=f$. 
Consequently \eqref{E1_p_and_q} holds. 
\end{proof}

\begin{lem}\label{l2_p_and_q}
Given a 2-structure $\sigma$, consider $X\subsetneq V(\sigma)$ such that $\sigma[X]$ is prime. 
Suppose that Statement (S3) holds. 
Consider distinct $e,f\in E(\sigma)$, and $\alpha\in X$. 
If $\Gamma_{(\sigma,\overline{X})}$ does not have isolated vertices, then 
the following two assertions hold 
\begin{enumerate}
\item if $\langle X\rangle_\sigma^{(e,f)}\neq\emptyset$ and 
$\langle X\rangle_\sigma^{(f,e)}\neq\emptyset$, then 
$\langle X\rangle_\sigma^{(e,f)}\longleftrightarrow_\sigma \langle X\rangle_\sigma^{(f,e)}$, and 
$$[\langle X\rangle_\sigma^{(e,f)},\langle X\rangle_\sigma^{(f,e)}]_\sigma=(e,f);$$
\item if $X_\sigma^{(e,f)}(\alpha)\neq\emptyset$ and $X_\sigma^{(f,e)}(\alpha)\neq\emptyset$, then 
$X_\sigma^{(e,f)}(\alpha)\longleftrightarrow_\sigma X_\sigma^{(f,e)}(\alpha)$, and 
$$[X_\sigma^{(e,f)}(\alpha),X_\sigma^{(f,e)}(\alpha)]_\sigma=(e,f).$$
\end{enumerate}
\end{lem}

\begin{proof}
For the first assertion, consider $x\in\langle X\rangle_\sigma^{(e,f)}$ and 
$x'\in\langle X\rangle_\sigma^{(f,e)}$. 
Since $x,x'\in\langle X\rangle_\sigma$, we have $\{x,x'\}\not\in E(\Gamma_{(\sigma,\overline{X})})$ by Remark~\ref{rem_p}. 
Furthermore, 
since $\Gamma_{(\sigma,\overline{X})}$ does not have isolated vertices, 
there exists $y'\in\overline{X}\setminus\{x,x'\}$ such that 
$\{x',y'\}\in E(\Gamma_{(\sigma,\overline{X})})$. 
Suppose for a contradiction that $\{x,y'\}\in E(\Gamma_{(\sigma,\overline{X})})$. 
We obtain $\{x,y'\},\{x',y'\}\in E(\Gamma_{(\sigma,\overline{X})})$. 
It follows from Fact~\ref{fac1_first_results} that $e=f$, which contradicts our assumption. 
Therefore $\{x,y'\}\not\in E(\Gamma_{(\sigma,\overline{X})})$. 
It follows from the first assertion of Fact~\ref{fac2_first_results} applied to $x,x',y'$ that 
$X\cup\{x',y'\}$ is a module of $\sigma[X\cup\{x,x',y'\}]$. 
Since $x\in\langle X\rangle_\sigma^{(e,f)}$, we obtain $(x,x')\in e$ and $(x',x)\in f$. 
Consequently $[x,x']_\sigma=(e,f)$. 

For the second assertion, consider $x\in X_\sigma^{(e,f)}(\alpha)$ and 
$x'\in X_\sigma^{(f,e)}(\alpha)$. 
Since $x,x'\in X_\sigma(\alpha)$, we have $\{x,x'\}\not\in E(\Gamma_{(\sigma,\overline{X})})$ by Remark~\ref{rem_p}. 
Furthermore, 
since $\Gamma_{(\sigma,\overline{X})}$ does not have isolated vertices, 
there exists $y'\in\overline{X}\setminus\{x,x'\}$ such that 
$\{x',y'\}\in E(\Gamma_{(\sigma,\overline{X})})$. 
Suppose for a contradiction that $\{x,y'\}\in E(\Gamma_{(\sigma,\overline{X})})$. 
We obtain $\{x,y'\},\{x',y'\}\in E(\Gamma_{(\sigma,\overline{X})})$. 
It follows from Fact~\ref{fac1_first_results} that $e=f$, which contradicts our assumption. 
Therefore $\{x,y'\}\not\in E(\Gamma_{(\sigma,\overline{X})})$. 
It follows from the second assertion of Fact~\ref{fac2_first_results} applied to $x,x',y'$ that 
$\{\alpha,x\}$ is a module of $\sigma[X\cup\{x,x',y'\}]$. 
Thus $[x,x']_\sigma=[\alpha,x']_\sigma$. 
Since $x'\in X_\sigma^{(f,e)}(\alpha)$, we obtain $[x,x']_\sigma=(e,f)$. 
\end{proof}

To state the next result, we use the following notation and definition. 

\begin{nota}\label{nota_restriction}
Let $\sigma$ be a 2-structure. 
For $e\in E(\sigma)$ and $W\subseteq V(\sigma)$, set 
$$e[W]=e\cap(W\times W).$$
Given $e\in E(\sigma)$ and $W\subseteq V(\sigma)$, we do not have 
$e\in E(\sigma[W])$, but we have $e[W]\in E(\sigma[W])$ when 
$e[W]\neq\emptyset$. 
\end{nota}

\begin{defi}\label{defi_linear}
A 2-structure $\sigma$ is {\em constant} if $|E(\sigma)|=1$. 
Besides, a 2-structure $\sigma$ is {\em linear} if 
there exist distinct $e,f\in E(\sigma)$ such that 
$$(V(\sigma),\{(v,w):v,w\in V(\sigma),v\neq w,[v,w]_\sigma=(e,f)\})$$ 
is a linear order (see Remark~\ref{rem_linear}). 
\end{defi}

\begin{rem}\label{rem_linear}
Let $\sigma$ be a linear 2-structure. 
There exist distinct $e,f\in E(\sigma)$ such that 
$(V(\sigma),\{(v,w):v,w\in V(\sigma),v\neq w,[v,w]_\sigma=(e,f)\})$ is a linear order. 
Therefore, $(V(\sigma),e)$ and $(V(\sigma),f)$ are total orders such that 
$$(V(\sigma),e)^\star=(V(\sigma),f).$$
Clearly, we have $E(\sigma)=\{e,f\}$. 
\end{rem}

\begin{lem}\label{l3_p_and_q}
Given a 2-structure $\sigma$, consider $X\subsetneq V(\sigma)$ such that $\sigma[X]$ is prime. 
Suppose that Statement (S3) holds. 
If $\sigma$ is prime, then the next two assertions hold. 
\begin{enumerate}
\item Let $e\in E(\sigma)$. 
If $|\langle X\rangle_\sigma^{(e,e)}|\geq 2$, then 
$\sigma[\langle X\rangle_\sigma]$ is constant, and 
$$E(\sigma[\langle X\rangle_\sigma])=\{e[\langle X\rangle_\sigma]\}.$$ 
Similarly, given $\alpha\in X$, if $|X_\sigma^{(e,e)}(\alpha)|\geq 2$, then 
$\sigma[X_\sigma(\alpha)]$ is constant, and 
$E(\sigma[X_\sigma(\alpha)])=\{e[X_\sigma(\alpha)]\}$. 
\item Consider distinct $e,f\in E(\sigma)$. 
If $|\langle X\rangle_\sigma^{(e,f)}|\geq 2$, then 
$\sigma[\langle X\rangle_\sigma]$ is linear, and 
$$E(\sigma[\langle X\rangle_\sigma])=\{e[\langle X\rangle_\sigma],f[\langle X\rangle_\sigma]\}.$$ 
Similarly, given $\alpha\in X$, if $|X_\sigma^{(e,f)}(\alpha)|\geq 2$, then 
$\sigma[X_\sigma(\alpha)]$ is linear, and 
$E(\sigma[X_\sigma(\alpha)])=\{e[X_\sigma(\alpha)],f[X_\sigma(\alpha)]\}$. 
\end{enumerate}
\end{lem}

\begin{proof}
Consider $B_q\in q_{(\sigma,\overline{X})}$, with $|B_q|\geq 2$. 
There exist $e,f\in E(\sigma)$ such that $B_q=\langle X\rangle_\sigma^{(e,f)}$ or 
$X_\sigma^{(e,f)}(\alpha)$. 
We define on $B_q$ the equivalence relation $\thickapprox$ 
in the following way. 
Given $c,d\in B_q$, 
$c\thickapprox d$ if either $c=d$ or $c\neq d$ and there exist sequences 
$(c_0,\ldots,c_p)$ and $(d_0,\ldots,d_q)$ of elements of $B_q$ satisfying
\begin{itemize}
\item $c_0=c$ and $c_p=d$;
\item for $0\leq m\leq p-1$, $[c_m,c_{m+1}]_\sigma\neq (e,f)$; 
\item $d_0=y$ and $d_q=x$;
\item for $0\leq m\leq q-1$, $[d_m,d_{m+1}]_\sigma\neq (e,f)$. 
\end{itemize}
Let us consider an equivalence class $C$ of $\thickapprox$. 
We prove that $C$ is a module of $\sigma$. 
We utilize Corollary~\ref{cor1_oppo_modules} in the following manner. 
Since $B_q\in q_{(\sigma,\overline{X})}$, there exists 
$B_p\in p_{(\sigma,\overline{X})}$ such that $B_q\subseteq B_p$. 

First, we show that $C$ is a module of $\sigma[B_q]$. 
Let $x\in B_q\setminus C$. 
By definition of $\thickapprox$, $[x,c]_\sigma=(e,f)$ or $(f,e)$ for every $c\in C$. 
Hence, $C$ is a module of $\sigma[B_q]$ when $e=f$. 
Suppose that $e\neq f$. 
For a contradiction, suppose that there exist $c,d\in C$ such that 
$[x,c]_\sigma=(e,f)$ and $[x,d]_\sigma=(f,e)$. 
Since $c\thickapprox d$, there exists a sequence 
$c_0,\ldots,c_p$ of elements of $B_q$ satisfying
\begin{itemize}
\item $d_0=d$ and $d_q=c$;
\item for $0\leq m\leq q-1$, $[d_m,d_{m+1}]_\sigma\neq (e,f)$.
\end{itemize}
By considering the sequences $(d=d_0,\ldots,d_q=c,x)$ and $(x,d)$, we obtain 
$x\thickapprox d$, which contradicts the fact that $C$ is an equivalence class of 
$\thickapprox$. 
It follows that $[x,C]_\sigma=(e,f)$ or $(f,e)$. 
Thus, $C$ is a module of $\sigma[B_q]$ when $e\neq f$. 

Second, we show that $C$ is a module of $\sigma[B_p]$. 
Suppose that $e=f$. 
It follows from Lemma~\ref{l1_p_and_q} that $B_q=B_p$. 
Hence $C$ is a module of $\sigma[B_p]$. 
Suppose that $e\neq f$. 
If $B_q=B_p$, then we proceed as previously. 
 Hence suppose that $B_q\neq B_p$. 
It follows from Lemma~\ref{l1_p_and_q} that 
$B_p\setminus B_q\in q_{(\sigma,\overline{X})}$ and 
\begin{equation*}
B_p\setminus B_q=\ 
\begin{cases}
\text{$\langle X\rangle_\sigma^{(f,e)}$ if $B_q=\langle X\rangle_\sigma^{(e,f)}$}\\
\text{or}\\
\text{$X_\sigma^{(f,e)}(\alpha)$ if $B_q=X_\sigma^{(e,f)}(\alpha)$}.
\end{cases}
\end{equation*}
It follows from Lemma~\ref{l2_p_and_q} that $B_q$ is a module of $\sigma[B_p]$. 
Since $C$ is a module of $\sigma[B_q]$, we obtain that 
$C$ is a module of $\sigma[B_p]$. 

Third, we prove that $C$ is a module of $\Gamma_{(\sigma,\overline{X})}$. 
Since $C\subseteq B_p$, we have 
$\{c,x\}\not\in E(\Gamma_{(\sigma,\overline{X})})$ for $c\in C$ and 
$x\in B_p\setminus C$ (see Remark~\ref{rem_p}). 
Therefore, we have to verify that 
$C$ is a module of $\Gamma_{(\sigma,\overline{X})}[C\cup\{x\}]$ 
for each $x\in\overline{X}\setminus B_p$. 
Let $x\in\overline{X}\setminus B_p$. 
Set 
\begin{equation*}
\begin{cases}
C^+=\{c\in C:\{c,x\}\in E(\Gamma_{(\sigma,\overline{X})})\}\\
{\rm and}\\
C^-=\{c\in C:\{c,x\}\not\in E(\Gamma_{(\sigma,\overline{X})})\}.
\end{cases}
\end{equation*}
For a contradiction, suppose that $C^-\neq\emptyset$ and $C^+\neq\emptyset$. 
It follows from Corollary~\ref{cor0_first_results} that $[C^-,C^+]_\sigma=(e,f)$ or 
$(f,e)$, which contradicts the fact that $C$ is an equivalence class of 
$\thickapprox$. 
Therefore, $C^-=\emptyset$ or $C^+=\emptyset$, that is, 
$C$ is a module of $\Gamma_{(\sigma,\overline{X})}[C\cup\{x\}]$ 
for each $x\in\overline{X}\setminus B_p$. 
Thus 
$C$ is a module of $\Gamma_{(\sigma,\overline{X})}$. 

Consequently, $C$ is a module of $\sigma[B_p]$, and $C$ is a module of $\Gamma_{(\sigma,\overline{X})}$. 
It follows from Corollary~\ref{cor1_oppo_modules} that $C$ is a module of 
$\sigma$. 
Since $\sigma$ is prime, $C$ is trivial. 
Hence $|C|=1$ because $C\neq\emptyset$, and $C\cap X=\emptyset$. 
We conclude as follows by distinguishing the following two cases. 
\begin{itemize}
\item Suppose that $e=f$. 
Recall that $B_q=B_p$ by Lemma~\ref{l1_p_and_q}. 
Since every equivalence class of $\thickapprox$ is reduced to a singleton, we obtain 
$(v,w)_\sigma=e$ for distinct elements $v$ and $w$ of $B_p$. 
In other words, 
$\sigma[B_p]$ is constant, and 
$E(\sigma[B_p])=\{e[B_p]\}$. 
\item Suppose that $e\neq f$. 
For instance, suppose that $B_q=\langle X\rangle_\sigma^{(e,f)}$. 
We verify that 
$\sigma[\langle X\rangle_\sigma^{(e,f)}]$ is linear, and 
$E(\sigma[\langle X\rangle_\sigma^{(e,f)}])=\{e[\langle X\rangle_\sigma^{(e,f)}],f[\langle X\rangle_\sigma^{(e,f)}]\}$.  
Since every equivalence class of $\thickapprox$ is reduced to a singleton, we obtain 
$[v,w]_\sigma=(e,f)$ or $(f,e)$ for distinct elements $v$ and $w$ of 
$\langle X\rangle_\sigma^{(e,f)}$. 
We consider the digraph $\lambda$ defined on $\langle X\rangle_\sigma^{(e,f)}$ as follows. 
Given distinct $v,w\in\langle X\rangle_\sigma^{(e,f)}$, $(v,w)\in A(\lambda)$ if 
$[v,w]_\sigma=(e,f)$. 
Since $[v,w]_\sigma=(e,f)$ or $(f,e)$ for distinct elements $v$ and $w$ of 
$\langle X\rangle_\sigma^{(e,f)}$, $\lambda$ is a tournament. 
For a contradiction, suppose that there exist distinct 
$u,v,w\in\langle X\rangle_\sigma^{(e,f)}$ such that $(u,v),(v,w),(w,u)\in A(\lambda)$. 
By considering the sequences $(v,u)$ and $(u,w,v)$, we obtain 
$v\thickapprox u$, which contradicts the fact that 
every equivalence class of $\thickapprox$ is reduced to a singleton. 
It follows that for distinct elements 
$u,v,w\in\langle X\rangle_\sigma^{(e,f)}$, if $(u,v),(v,w)\in A(\lambda)$, then 
$(u,w)\in A(\lambda)$. 
Therefore, $\lambda$ is a linear order, that is, 
$\sigma[\langle X\rangle_\sigma^{(e,f)}]$ is linear, and 
$E(\sigma[\langle X\rangle_\sigma^{(e,f)}])=\{e[\langle X\rangle_\sigma^{(e,f)}],f[\langle X\rangle_\sigma^{(e,f)}]\}$.  

Lastly, suppose that $B_q\subsetneq B_p$. 
It follows from Lemma~\ref{l1_p_and_q} that 
$B_p\setminus B_q=\langle X\rangle_\sigma^{(f,e)}$. 
Similarly, we have $\sigma[\langle X\rangle_\sigma^{(f,e)}]$ is linear, and 
$E(\sigma[\langle X\rangle_\sigma^{(f,e)}])=\{e[\langle X\rangle_\sigma^{(f,e)}],f[\langle X\rangle_\sigma^{(f,e)}]\}$.   
Moreover, we have 
$$[\langle X\rangle_\sigma^{(e,f)},\langle X\rangle_\sigma^{(f,e)}]_\sigma=(e,f)$$ 
by the first assertion of Lemma~\ref{l2_p_and_q}. 
Consequently, 
$\sigma[\langle X\rangle_\sigma]$ is linear, and 
$E(\sigma[\langle X\rangle_\sigma])=\{e[\langle X\rangle_\sigma],f[\langle X\rangle_\sigma]\}$. \qedhere
\end{itemize}
\end{proof}

Lemma~\ref{l3_p_and_q} ends the examination of blocks of the partitions 
$p_{(\sigma,\overline{X})}$ and $q_{(\sigma,\overline{X})}$. 
We complete Section~\ref{S_first_results} with a result on the components of the outside graph, which follows from Fact~\ref{fac1_first_results} and the following easy consequence of Fact~\ref{fac2_first_results}. 
We use the following notation. 

\begin{nota}\label{nota_a+s}
Given a 2-structure $\sigma$, consider $X\subsetneq V(\sigma)$ such that $\sigma[X]$ is prime. 
First, the set $\{\langle X\rangle_\sigma^{(e,e)}:e\in E(\sigma)\}\cup\{X_\sigma^{(e,e)}(\alpha):e\in E(\sigma),\alpha\in X\}$ is denoted by 
$q_{(\sigma,\overline{X})}^s$. 
Second, the set $q_{(\sigma,\overline{X})}\setminus(q_{(\sigma,\overline{X})}^s\cup\{{\rm Ext}_\sigma(X)\})$ is denoted by $q_{(\sigma,\overline{X})}^a$. 
\end{nota}

\begin{fac}\label{fac3_first_results}
Given a 2-structure $\sigma$, consider $X\subsetneq V(\sigma)$ such that $\sigma[X]$ is prime. 
Suppose that Statement (S3) holds. 
Consider distinct elements $x,x',y,y'$ of $\overline{X}$ such that 
$\{x,y\},\{x',y'\}\in E(\Gamma_{(\sigma,\overline{X})})$ and 
$\{x,y'\},\{x',y\}\not\in E(\Gamma_{(\sigma,\overline{X})})$. 
If there exist $B_q\in q_{(\sigma,\overline{X})}$ such that $y,y'\in B_q$, then 
$B_q\in q_{(\sigma,\overline{X})}^s$. 
\end{fac}

\begin{proof}
Since $y$ and $y'$ belong to the same block of $p_{(\sigma,\overline{X})}$, we have 
$\{y,y'\}\not\in E(\Gamma_{(\sigma,\overline{X})})$ by Remark~\ref{rem_p}. 
Besides, there exist $e,f\in E(\sigma)$ such that $B_q=\langle X\rangle_\sigma^{(e,f)}$ or 
$B_q=X_\sigma^{(e,f)}(\alpha)$, where $\alpha\in X$. 

First, suppose that $B_q=\langle X\rangle_\sigma^{(e,f)}$. 
By the first assertion of Fact~\ref{fac2_first_results} applied to $x,y,y'$, $X\cup\{x,y\}$ is a module of $\sigma[X\cup\{x,y,y'\}]$. 
Since $y'\in\langle X\rangle_\sigma^{(e,f)}$, $[y',y]_\sigma=(e,f)$. 
Similarly, it follows from the first assertion of Fact~\ref{fac2_first_results} applied to $x',y,y'$ that 
$[y,y']_\sigma=(e,f)$. 
Thus $e=f$, and hence $B_q\in q_{(\sigma,\overline{X})}^s$. 

Second, suppose that $B_q=X_\sigma^{(e,f)}(\alpha)$, where $\alpha\in X$. 
By the second assertion of Fact~\ref{fac2_first_results} applied to $x,y,y'$, $\{\alpha,y'\}$ is a module of $\sigma[X\cup\{x,y,y'\}]$. 
Thus $[y,y']_\sigma=[y,\alpha]_\sigma$. 
Since $y\in X_\sigma^{(e,f)}(\alpha)$, we obtain $[y,y']_\sigma=(e,f)$. 
Similarly, it follows from the second assertion of Fact~\ref{fac2_first_results} applied to $x',y,y'$ that 
$[y',y]_\sigma=(e,f)$. 
Therefore $e=f$, so $B_q\in q_{(\sigma,\overline{X})}^s$. 
\end{proof}

\begin{prop}\label{prop_component}
Given a 2-structure $\sigma$, consider $X\subsetneq V(\sigma)$ such that $\sigma[X]$ is prime. 
Suppose that Statement (S3) holds. 
If $\Gamma_{(\sigma,\overline{X})}$ does not have isolated vertices, then the following two assertions hold. 
\begin{enumerate}
\item For each component $C$ of $\Gamma_{(\sigma,\overline{X})}$, there exist distinct 
$B_p,D_p\in p_{(\sigma,\overline{X})}$ and $B_q,D_q\in q_{(\sigma,\overline{X})}$ such that 
$B_q\subseteq B_p$, $D_q\subseteq D_p$, and $C$ is bipartite with bipartition 
$\{V(C)\cap B_q,V(C)\cap D_q\}$. 
\item For a component $C$ of $\Gamma_{(\sigma,\overline{X})}$ and for 
$B_q\in q_{(\sigma,\overline{X})}^a$, if $V(C)\cap B_q\neq\emptyset$, then $B_q\subseteq V(C)$. 
\end{enumerate}
\end{prop}

\begin{proof}
For the first assertion, consider a component $C$ of 
$\Gamma_{(\sigma,\overline{X})}$. 
Since $\Gamma_{(\sigma,\overline{X})}$ does not have isolated vertices, 
$v(C)\geq 2$. 
Hence, there exist distinct $c,d\in V(C)$ such that 
$\{c,d\}\in E(\Gamma_{(\sigma,\overline{X})})$. 
There exist $B_p,D_p\in p_{(\sigma,\overline{X})}$ and $B_q,D_q\in q_{(\sigma,\overline{X})}$ such that $c\in B_q$, $d\in D_q$, $B_q\subseteq B_p$ and $D_q\subseteq D_p$. 
Since $\{c,d\}\in E(\Gamma_{(\sigma,\overline{X})})$, we have $B_p\neq D_p$ by Remark~\ref{rem_p}. 
Let $x\in V(C)\setminus\{c,d\}$. 
Since $C$ is a component of $\Gamma_{(\sigma,\overline{X})}$, there exist a path 
$x_0\ldots x_n$ such that $x_0\in\{c,d\}$, $x_n=x$, and $\{x_1,\ldots,x_n\}\cap\{c,d\}=\emptyset$. 
We have $n\geq 1$. 
We distinguish the following two cases. 
\begin{enumerate}
\item Suppose that $n$ is even. 
It follows from Fact~\ref{fac1_first_results} that $x_0,x_2,\ldots,x_n$ belong to the same block of 
$q_{(\sigma,\overline{X})}$. 
Since $x_0\in\{c,d\}$ and $x_n=x$, we obtain $x\in B_q\cup D_q$. 
\item Suppose that $n$ is odd. 
Set 
\begin{equation*}
x_{-1}=\ 
\begin{cases}
d\ {\rm if}\ x_0=c\\
{\rm and}\\
c\ {\rm if}\ x_0=d.
\end{cases}
\end{equation*}
We have $x_{-1}\in B_q\cup D_q$. 
By considering the path $x_{-1}x_0\ldots x_n$, it follows from Fact~\ref{fac1_first_results} that 
$x$ and $x_{-1}$ belong to the same block of $q_{(\sigma,\overline{X})}$. 
Hence $x\in B_q\cup D_q$. 
\end{enumerate}
Therefore $V(C)\setminus\{c,d\}\subseteq B_q\cup D_q$, so 
$V(C)\subseteq B_q\cup D_q$. 
By Remark~\ref{rem_p}, 
$C$ is bipartite with bipartition 
$\{V(C)\cap B_q,V(C)\cap D_q\}$. 

For the second assertion, consider a component $C$ of $\Gamma_{(\sigma,\overline{X})}$, and an  element $B_q$ of $q_{(\sigma,\overline{X})}^a$ such that $V(C)\cap B_q\neq\emptyset$. 
Consider $y\in V(C)\cap B_q$. 
For a contradiction, suppose that $B_q\setminus V(C)\neq\emptyset$, and consider 
$y'\in B_q\setminus V(C)$. 
Since $\Gamma_{(\sigma,\overline{X})}$ does not have isolated vertices, there exist $x\in\overline{X}\setminus\{y\}$ and 
$x'\in\overline{X}\setminus\{y'\}$ such that 
$\{x,y\},\{x',y'\}\in E(\Gamma_{(\sigma,\overline{X})})$. 
Furthermore, since $C$ is a component of $\Gamma_{(\sigma,\overline{X})}$, with 
$y\in V(C)$ and $y'\not\in V(C)$, 
we obtain $x\in V(C)$ and $x'\not\in V(C)$. 
Hence $x\neq x'$, and $\{x,y'\},\{x',y\}\not\in E(\Gamma_{(\sigma,\overline{X})})$. 
It follows from Fact~\ref{fac3_first_results} that $B_q\in q_{(\sigma,\overline{X})}^s$, which contradicts $B_q\in q_{(\sigma,\overline{X})}^a$. 
\end{proof}

Proposition~\ref{prop_component} leads us to the following notation. 

\begin{nota}\label{nota_prop_component}
Given a 2-structure $\sigma$, consider $X\subsetneq V(\sigma)$ such that $\sigma[X]$ is prime. 
Suppose that Statement (S3) holds. 
To use Proposition~\ref{prop_component}, we have also to suppose that 
$\Gamma_{(\sigma,\overline{X})}$ does not have isolated vertices. 
By Corollary~\ref{cor1_first_results}, we can also suppose that $\sigma$ is prime. 

Consider a component $C$ of $\Gamma_{(\sigma,\overline{X})}$. 
By the first assertion of Proposition~\ref{prop_component}, 
there exist distinct 
$B_p,D_p\in p_{(\sigma,\overline{X})}$ and $B_q,D_q\in q_{(\sigma,\overline{X})}$ such that 
$B_q\subseteq B_p$, $D_q\subseteq D_p$, and $C$ is bipartite with bipartition 
$\{V(C)\cap B_q,V(C)\cap D_q\}$. 
In the sequel, $V(C)\cap B_q$ and $V(C)\cap D_q$ are respectively denoted by 
$B^C_q$ and $D^C_q$. (Note that we use the Axiom of Ultrafilter to introduce such a notation for each component of $\Gamma_{(\sigma,\overline{X})}$, when 
$q_{(\sigma,\overline{X})}$ has infinitely many blocks.) 
\end{nota}

\section{Proofs of the main results}

We use the following notation. 

\begin{nota}\label{nota_component}
Given a graph $\Gamma$, $\mathcal{C}(\Gamma)$ denotes the set of the components of $\Gamma$.  
\end{nota}

\begin{proof}[Proof of Theorem~\ref{thm_main_1}]
To begin, suppose that $\sigma$ is not prime. 
We prove that there exists $C\in\mathcal{C}(\Gamma_{(\sigma,\overline{X})})$ such that $\sigma[X\cup V(C)]$ is not prime. 
First, suppose that $\Gamma_{(\sigma,\overline{X})}$ admits isolated vertices. 
There exists $v\in\overline{X}$ such that $\{v\}\in\mathcal{C}(\Gamma_{(\sigma,\overline{X})})$. 
Since Statement (S3) holds, ${\rm Ext}_\sigma(X)=\emptyset$ by 
Remark~\ref{rem1_conditions_Ck}. 
Thus $\sigma[X\cup\{v\}]$ is not prime. 
Second, suppose that $\Gamma_{(\sigma,\overline{X})}$ does not have isolated vertices. 
Since $\sigma$ is not prime, $\sigma$ admits a nontrivial module $M$. 
It follows from Corollary~\ref{cor1_modules} that $M\cap X=\emptyset$. 
By Lemma~\ref{lem2_modules}, there exists $B_p\in p_{(\sigma,\overline{X})}$ such that 
$M\subseteq B_p$, and $M$ is a module of $\Gamma_{(\sigma,\overline{X})}$. 
Let $x\in M$. 
Since $\Gamma_{(\sigma,\overline{X})}$ does not have isolated vertices, there exists $y\in\overline{X}\setminus\{x\}$ such that 
$\{x,y\}\in E(\Gamma_{(\sigma,\overline{X})})$. 
Since $M\subseteq B_p$, we have $y\not\in M$ by Remark~\ref{rem_p}. 
Denote by $C$ the component of $\Gamma_{(\sigma,\overline{X})}$ 
containing $x$. 
Hence $y\in V(C)$ because $\{x,y\}\in E(\Gamma_{(\sigma,\overline{X})})$. 
Since $M$ is a module of $\Gamma_{(\sigma,\overline{X})}$, we obtain 
$\{x',y\}\in E(\Gamma_{(\sigma,\overline{X})})$ for every $x'\in M$. 
Therefore $M\subseteq V(C)$. 
It follows that $M$ is a nontrivial module of $\sigma[X\cup V(C)]$. 

Now, we suppose that there exists $C\in\mathcal{C}(\Gamma_{(\sigma,\overline{X})})$ such that $\sigma[X\cup V(C)]$ is not prime. 
Since $\sigma[X\cup V(C)]$ is not prime, we have $v(C)\neq 2$. 
We assume that $v(C)\geq 4$, and we have to prove that $C$ is not prime. 
Consider a nontrivial module $M$ of $\sigma[X\cup V(C)]$. 
Clearly, $\sigma[X\cup V(C)]$ satisfies Statement~(S3). 
Moreover, $$\Gamma_{(\sigma[X\cup V(C)],V(C))}=C.$$
Since $v(C)\geq 4$, it follows from Corollary~\ref{cor1_modules} applied to 
$\sigma[X\cup V(C)]$ that $M\subseteq V(C)$. 
By Lemma~\ref{lem2_modules} applied to $\sigma[X\cup V(C)]$, there exists 
$B_p\in p_{(\sigma[X\cup V(C)],V(C))}$ such that 
$M\subseteq B_p^C$, and $M$ is a module of $C$. 
We have to verify that $M\neq V(C)$. 
Let $x\in M$. 
Since $v(C)\geq 4$, there exists 
$y\in V(C)\setminus\{x\}$ such that $\{x,y\}\in E(C)$. 
Since $M\subseteq B_p^C$, we have $y\not\in M$ by 
Remark~\ref{rem_p} applied to $\sigma[X\cup V(C)]$. 
Hence $y\in V(C)\setminus M$. 

Lastly, we suppose that there exists 
$C\in\mathcal{C}(\Gamma_{(\sigma,\overline{X})})$ such that $v(C)=1$ or $v(C)\geq 3$ and $C$ is not prime. 
We have to prove that $\sigma$ is not prime. 
Therefore, by Corollary~\ref{cor1_first_results}, we can assume that 
\begin{equation}\label{E1_thm_main_1}
\text{$\Gamma_{(\sigma,\overline{X})}$ does not have isolated vertices.}
\end{equation}
In particular, we obtain $v(C)\geq 3$. 
Consider a nontrivial module $M$ of $C$. 
Clearly, $M$ is a module of $\Gamma_{(\sigma,\overline{X})}$ because $C$ is a component of 
$\Gamma_{(\sigma,\overline{X})}$. 
Since $\Gamma_{(\sigma,\overline{X})}$ does not have isolated vertices 
(see \eqref{E1_thm_main_1}), it follows from the first assertion of Proposition~\ref{prop_component} that there exist distinct 
$B_p,D_p\in p_{(\sigma,\overline{X})}$ and $B_q,D_q\in q_{(\sigma,\overline{X})}$ such that 
$B_q\subseteq B_p$, $D_q\subseteq D_p$, and $C$ is bipartite with bipartition 
$\{V(C)\cap B_q,V(C)\cap D_q\}$. 
Since $C$ is connected, we have $M\subseteq V(C)\cap B_q$ or $M\subseteq V(C)\cap D_q$. 
For instance, assume that $M\subseteq V(C)\cap B_q$. 
To conclude, we distinguish the following two cases. 
\begin{enumerate}
\item Suppose that $B_q\in q_{(\sigma,\overline{X})}^s$. 
There exists $e\in E(\sigma)$ such that $B_q=\langle X\rangle_\sigma^{(e,e)}$ or 
$X_\sigma^{(e,e)}(\alpha)$, where $\alpha\in X$. 
If $\sigma[B_p]$ is not constant, then it follows from the first assertion of Lemma~\ref{l3_p_and_q} 
that $\sigma$ is not prime. 
Thus, suppose that $\sigma[B_p]$ is constant. 
It follows that any subset of $B_p$ is a module of $\sigma[B_p]$. 
In particular, $M$ is a module of $\sigma[B_p]$. 
Since $M$ is a module of $\Gamma_{(\sigma,\overline{X})}$, 
it follows from Corollary~\ref{cor1_oppo_modules} that $M$ is a module of $\sigma$. 
\item Suppose that $B_q\in q_{(\sigma,\overline{X})}^a$. 
Since $\Gamma_{(\sigma,\overline{X})}$ does not have isolated vertices 
(see \eqref{E1_thm_main_1}), it follows from the second assertion of Proposition~\ref{prop_component} that $B_q\subseteq V(C)$. 
In general, $M$ is not a module of $\sigma[B_q]$, and hence $M$ is not a module of 
$\sigma[B_p]$. 
Therefore, we cannot apply Corollary~\ref{cor1_oppo_modules} to $M$. 
Nevertheless, we construct a superset of $M$, which is a module of 
$\Gamma_{(\sigma,\overline{X})}$, and a module of 
$\sigma[B_p]$. 
Consider the set $\mathcal{M}$ of the nontrivial modules $M'$ of $C$ such that $M\subseteq M'$. 
Set $$\widetilde{M}=\bigcup\mathcal{M}.$$
Clearly, $M\in\mathcal{M}$. 
Since $M\neq\emptyset$ and all the elements of $\mathcal{M}$ contain $M$, 
$\widetilde{M}$ is a module of $C$. 
Since $C$ is a component of $\Gamma_{(\sigma,\overline{X})}$, 
$\widetilde{M}$ is a module of $\Gamma_{(\sigma,\overline{X})}$. 
As previously seen for $M$, $\widetilde{M}\subseteq V(C)\cap B_q$ or 
$\widetilde{M}\subseteq V(C)\cap D_q$. 
Since $M\subseteq\widetilde{M}$ and $M\subseteq V(C)\cap B_q$, we have 
$\widetilde{M}\subseteq V(C)\cap B_q$. 
Therefore $\widetilde{M}\subseteq B_q$. 
Set $$N=\{v\in B_q\setminus\widetilde{M}:v\not\longleftrightarrow_\sigma \widetilde{M}\}.$$
We verify that $\widetilde{M}\cup N$ is a module of $C$. 
It suffices to show that for any $v\in V(C)\cap D_q$, $x\in\widetilde{M}$ and $y\in N$, we have 
$\{v,x\},\{v,y\}\in E(\Gamma_{(\sigma,\overline{X})})$ or 
$\{v,x\},\{v,y\}\not\in E(\Gamma_{(\sigma,\overline{X})})$. 
Since $y\in N$, there exist $x',x''\in\widetilde{M}$ such that 
$y\not\longleftrightarrow_\sigma \{x',x''\}$. 
Furthermore, since $\widetilde{M}$ is a module of $C$, we have 
$\{v,x\},\{v,x'\},\{v,x''\}\in E(\Gamma_{(\sigma,\overline{X})})$ or 
$\{v,x\},\{v,x'\},\{v,x''\}\not\in E(\Gamma_{(\sigma,\overline{X})})$. 
For instance, suppose that 
$\{v,x\},\{v,x'\},\{v,x''\}\in E(\Gamma_{(\sigma,\overline{X})})$. 
By Corollary~\ref{cor0_first_results}, $\{z\in B_q:\{z,v\}\in E(\Gamma_{(\sigma,\overline{X})})\}$ is a module of $\sigma[B_q]$. 
Since $x,x',x''\in\{z\in B_q:\{z,v\}\in E(\Gamma_{(\sigma,\overline{X})})\}$ and 
$y\not\longleftrightarrow_\sigma \{x',x''\}$, we obtain 
$y\in\{z\in B_q:\{z,v\}\in E(\Gamma_{(\sigma,\overline{X})})\}$. 
Hence $\{v,x\},\{v,x'\},\{v,x''\},\{v,y\}\in E(\Gamma_{(\sigma,\overline{X})})$. 
Similarly, if $\{v,x\},\{v,x'\},\{v,x''\}\not\in E(\Gamma_{(\sigma,\overline{X})})$, then if follows from 
Corollary~\ref{cor0_first_results} that $\{v,x\},\{v,x'\},\{v,x''\},\{v,y\}\not\in E(\Gamma_{(\sigma,\overline{X})})$. 
Consequently, $\widetilde{M}\cup N$ is a module of $C$. 
It follows from the definition of $\widetilde{M}$ that $N\subseteq\widetilde{M}$. 
Therefore $N=\emptyset$, and hence $\widetilde{M}$ is a module of $\sigma[B_q]$. 
Since $\Gamma_{(\sigma,\overline{X})}$ does not have isolated vertices 
(see \eqref{E1_thm_main_1}), it follows from Lemmas~\ref{l1_p_and_q} and \ref{l2_p_and_q} that $\widetilde{M}$ is a module of 
$\sigma[B_p]$. 
Lastly, since $\widetilde{M}$ is a module of $\Gamma_{(\sigma,\overline{X})}$, 
it follows from Corollary~\ref{cor1_oppo_modules} that $\widetilde{M}$ is a module of $\sigma$. \qedhere
\end{enumerate}
\end{proof}

We use the next notation to demonstrate Theorem~\ref{thm_main_2}. 

\begin{nota}\label{nota_set_comp}
Given graphs $G$ and $H$, $G\leq H$ means that $G$ is isomorphic to an induced subgraph of $H$.  
\end{nota}

\begin{nota}\label{nota_union}
Let $G$ and $H$ be graphs such that $V(G)\cap V(H)=\emptyset$. 
The {\em disjoint union} of $G$ and $H$ is the graph 
$G\oplus H=(V(G)\cup V(H),E(G)\cup E(H))$. 
If $V(G)\cap V(H)\neq\emptyset$, then we can define $G\oplus H$ up to 
isomorphism by considering graphs $G'$ and $H'$ such that $G\simeq G'$, $H\simeq H'$, and $V(G')\cap V(H')=\emptyset$. 
\end{nota}

We use also the following two lemmas. 
The next result is a consequence of Theorem~\ref{thm_main_1}. 

\begin{lem}\label{lem1_component}
Given a 2-structure $\sigma$, consider $X\subsetneq V(\sigma)$ such that $\sigma[X]$ is prime. 
Suppose that Statement (S5) holds. 
For each component $C$ of $\Gamma_{(\sigma,\overline{X})}$, $P_5\not\leq C$. 
\end{lem}

\begin{proof}
For a contradiction, suppose that there exists a component $C$ of 
$\Gamma_{(\sigma,\overline{X})}$ such that $P_5\leq C$. 
Hence, there exists $Y\subseteq V(C)$ such that 
$C[Y]\simeq P_5$. 
We have 
$$\Gamma_{(\sigma[X\cup Y],Y)}=\Gamma_{(\sigma,\overline{X})}[Y].$$
Since $\Gamma_{(\sigma,\overline{X})}[Y]=C[Y]$, $\Gamma_{(\sigma[X\cup Y],Y)}$ is prime. 
It follows from Theorem~\ref{thm_main_1} applied to $\sigma[X\cup Y]$ that 
$\sigma[X\cup Y]$ is prime, which contradicts the fact that 
Statement (S5) holds. 
\end{proof}

Since the proof of the next lemma is easy, we omit it. 

\begin{lem}\label{fac1_half_graph}
Given a connected graph $\Gamma$, $K_2\oplus K_2\leq \Gamma$ if and only if $P_5\leq \Gamma$. 
\end{lem}

\begin{proof}[Proof of Theorem~\ref{thm_main_2}]
To begin, suppose that the first assertion holds, that is, $\sigma$ is $\overline{X}\!$-critical. 
We have to prove that the second assertion holds. 
Consider $C\in\mathcal{C}(\Gamma_{(\sigma,\overline{X})})$.  
By Theorem~\ref{thm_main_1} applied to $\sigma$, $\sigma[X\cup V(C)]$ is prime. 
We have to show that $\sigma[X\cup V(C)]$ is $V(C)$-critical. 
Let $c\in V(C)$. 
Since $\sigma$ is $\overline{X}\!$-critical, $\sigma-c$ is not prime. 
We have $$\Gamma_{(\sigma-c,\overline{X}\setminus\{c\})}=\Gamma_{(\sigma,\overline{X})}-c.$$
Therefore 
\begin{equation}\label{E1_thm_main_2}
\mathcal{C}(\Gamma_{(\sigma-c,\overline{X}\setminus\{c\})})=
(\mathcal{C}(\Gamma_{(\sigma,\overline{X})})\setminus\{C\})\cup\mathcal{C}(C-c).
\end{equation}
Since $\sigma-c$ is not prime, it follows from Theorem~\ref{thm_main_1} applied to 
$\sigma-c$ that there exists 
$C'\in\mathcal{C}(\Gamma_{(\sigma-c,\overline{X}\setminus\{c\})})$ such that 
$\sigma[X\cup V(C')]$ is not prime. 
By \eqref{E1_thm_main_2}, 
$C'\in(\mathcal{C}(\Gamma_{(\sigma,\overline{X})})\setminus\{C\})\cup\mathcal{C}(C-c)$. 
By Theorem~\ref{thm_main_1} applied to $\sigma$, $\sigma[X\cup V(D)]$ is prime for every $D\in\mathcal{C}(\Gamma_{(\sigma,\overline{X})})\setminus\{C\})$. 
Thus $C'\in\mathcal{C}(C-c)$. 
Finally, 
since $$\Gamma_{(\sigma[X\cup V(C)]-c,V(C)\setminus\{c\})}=C-c,$$
it follows from Theorem~\ref{thm_main_1} applied to $\sigma[X\cup V(C)]-c$ that 
$\sigma[X\cup V(C)]-c$ is not prime. 
Consequently $\sigma[X\cup V(C)]$ is $V(C)$-critical. 

To continue, suppose that the second assertion holds. 
We have to prove that the third assertion holds. 
Consider $C\in\mathcal{C}(\Gamma_{(\sigma,\overline{X})})$. 
By Theorem~\ref{thm_main_1} applied to $\sigma$, $v(C)=2$ or $v(C)\geq 4$ and 
$C$ is prime. 
Suppose that $v(C)\geq 4$ and 
$C$ is prime. 
We have to show that $C$ is critical. 
If $v(C)=4$, then $C$ is critical by Proposition~\ref{prop1_half_graph}. 
Hence suppose that $v(C)\geq 5$. 
Let $c\in V(C)$. 
If $C-c$ is disconnected, then $C-c$ is not prime. 
Thus, suppose that $C-c$ is connected. 
Since the second assertion holds, $\sigma[X\cup V(C)]-c$ is not prime. 
We have $$\Gamma_{(\sigma[X\cup V(C)]-c,V(C)\setminus\{c\})}=C-c.$$
It follows from Theorem~\ref{thm_main_1} applied to $\sigma[X\cup V(C)]-c$ that $C-c$ is not prime. 

Lastly, suppose that the third assertion holds. 
Hence, for every $C\in\mathcal{C}(\Gamma_{(\sigma,\overline{X})})$, 
\begin{equation}\label{E2_thm_main_2}
\text{$v(C)=2$ or $v(C)\geq 4$ and $C$ is critical.} 
\end{equation} 
We have to prove that $\sigma$ is $\overline{X}$-critical. 
By Theorem~\ref{thm_main_1} applied to $\sigma$, $\sigma$ is prime. 
Let $x\in\overline{X}$. 
We have to prove that $\sigma-x$ is not prime. 
Denote by $C$ the component of $\Gamma_{(\sigma,\overline{X})}$ containing $x$. 
As seen in \eqref{E1_thm_main_2}, 
\begin{equation}\label{E4_thm_main_2}
\mathcal{C}(C-x)\subseteq
\mathcal{C}(\Gamma_{(\sigma-x,\overline{X}\setminus\{x\})}).
\end{equation}
Suppose that $C-x$ admits isolated vertices. 
By \eqref{E4_thm_main_2}, 
$\Gamma_{(\sigma-x,\overline{X}\setminus\{x\})}$ admits isolated vertices as well. 
It follows from Corollary~\ref{cor1_first_results} that $\sigma-x$ is not prime. 
Finally, suppose that $C-x$ does not admit isolated vertices, that is, $v(C')\geq 2$ for each $C'\in\mathcal{C}(C-x)$. 
In particular, we do not have $v(C)=2$. 
It follows from \eqref{E2_thm_main_2} that 
\begin{equation}\label{E5_thm_main_2}
\text{$v(C)\geq 4$ and $C$ is critical.} 
\end{equation} 
By Lemma~\ref{lem1_component}, $P_5\not\leq C$. 
Therefore $K_2\oplus K_2\not\leq C$ by Lemma~\ref{fac1_half_graph}. 
Since $v(C')\geq 2$ for each $C'\in\mathcal{C}(C-x)$, we obtain that $C-x$ possesses a unique component, that is, $C-x$ is connected. 
By \eqref{E4_thm_main_2}, $C-x\in\mathcal{C}(\Gamma_{(\sigma-x,\overline{X}\setminus\{x\})})$. 
Furthermore, it follows from \eqref{E5_thm_main_2} that $v(C-x)\geq 3$ and 
$C-x$ is not prime. 
By Theorem~\ref{thm_main_1} applied to $\sigma-x$, $\sigma-x$ is not prime. 
\end{proof}

\begin{proof}[Proof of Corollary~\ref{cor1_thm_main_2}]
To begin, suppose that $\sigma$ is $\overline{X}\!$-critical. 
As seen in Remark~\ref{rem1_conditions_Ck}, Statement (S5) holds. 

Conversely, suppose that Statement (S5) holds, and $\sigma$ is prime. 
To prove that $\sigma$ is $\overline{X}\!$-critical, we apply 
Theorem~\ref{thm_main_2}. 
Let $C$ be a component of $\Gamma_{(\sigma,\overline{X})}$. 
Since $\sigma$ is prime, it follows from Theorem~\ref{thm_main_1} that $v(C)=2$ or $v(C)\geq 4$ and $C$ is prime. 
Suppose that $v(C)\geq 4$ and $C$ is prime. 
By Lemma~\ref{lem1_component}, $P_5\not\leq C$. 
It follows from Proposition~\ref{prop1_half_graph} that $C$ is critical. 
Consequently, for each component $C$ of $\Gamma_{(\sigma,\overline{X})}$, we have $v(C)=2$ or $v(C)\geq 4$ and $C$ is critical. 
By Theorem~\ref{thm_main_2}, $\sigma$ is $\overline{X}\!$-critical. 
\end{proof}

\begin{proof}[Proof of Theorem~\ref{cor1_thm_main_1}]
To begin, suppose that Statement (S5) holds, and $\sigma$ is prime. 
Let $F$ be a finite subset of $\overline{X}$. 
By Corollary~\ref{cor1_pi_infinite}, 
there exist a finite subset $F'$ of 
$\overline{X}$ such that 
$F\subseteq F'$ and $\sigma[X\cup F']$ is prime. 
Since Statement (S5) holds, it follows from Corollary~\ref{cor1_thm_main_2} that 
$\sigma[X\cup F']$ is $(F')$-critical. 
Consequently, $\sigma$ is finitely $\overline{X}$-critical. 

Conversely, suppose that  $\sigma$ is finitely $\overline{X}$-critical. 
Hence, we obtain that for each finite subset $F$ of $\overline{X}$, there exist a finite subset $F'$ of 
$\overline{X}$ such that 
$F\subseteq F'$ and $\sigma[X\cup F']$ is prime. 
By Corollary~\ref{cor1_pi_infinite}, $\sigma$ is prime. 
Lastly, consider $W\subseteq\overline{X}$ such that $|W|=5$. 
Since $\sigma$ is finitely $\overline{X}$-critical, there exists 
$W'\subseteq\overline{X}$ such that $W'$ is finite and $\sigma[X\cup W']$ is 
$(W')$-critical. 
As seen in Remark~\ref{rem1_conditions_Ck}, Statement (S5) holds in 
$\sigma[X\cup W']$. 
Therefore Statement (S5) holds in $\sigma$. 
\end{proof}

\section{Half graphs}\label{section_Half_graphs}

We begin with a remark on half graphs. 

\begin{rem}\label{rem_infinite_half_graph}
Consider a half graph $\Gamma$, with bipartition $\{X,Y\}$. 
There exist a linear order $L$ defined on $X$, and a bijection $\varphi$ from $X$ onto $Y$ such that $E(\Gamma)=\{\{x,\varphi(x')\}:x\leq x'\!\mod L\}$. 
Denote by $\varphi(L)$ the unique linear order defined on $Y$ such that $\varphi$ is an isomorphism from $L$ onto $\varphi(L)$. 
We obtain $$E(\Gamma)=\{\{y,\varphi^{-1}(y')\}:y\leq y'\hspace{-2mm}\mod \varphi(L)^\star\}.$$
Consequently, $\Gamma$ is also a half graph by considering the linear order $\varphi(L)^\star$ defined on $Y$, and the bijection $\varphi^{-1}:Y\longrightarrow X$. 
\end{rem}

In the next remark, we explain how to decompose a discrete linear order 
(see Definition~\ref{defi_discrete}) into a lexicographic sum. 

\begin{rem}\label{rem_discrete}
Given an infinite linear order $L$, $L$ is discrete if and only if $L$ is decomposed into a lexicographic sum $\sum_ll_v$ satisfying the following conditions.
\begin{enumerate}
\item If $l$ admits a unique vertex $v$, then $L=l_v$, and 
$L\simeq\omega^\star$ or $\omega$ or $\omega^\star+\omega$. 
\item For every $v\in V(l)$, if $v$ is neither the smallest nor the largest element of $l$, then $l_v\simeq\omega^\star+\omega$.
\item If $l$ admits a smallest element denoted by ${\rm min}$, then $l_{\rm min}\simeq\omega$ or $\omega^\star+\omega$.
\item If $l$ admits a largest element denoted by ${\rm max}$, then $l_{\rm max}\simeq\omega^\star$ or $\omega^\star+\omega$. 
\end{enumerate}

\begin{proof}[Hint]
For a linear order, both notions of an interval and a module coincide. 
Consider an infinite discrete linear order $L$. 
We define on $V(L)$ the binary relation $\sim$ as follows. 
Given $v,w\in V(L)$, $v\sim w$ if the smallest interval of $L$ containing $v$ and $w$ is finite. 
Clearly, $\sim$ is an equivalence relation. 
Furthermore, the equivalence classes of $\sim$ are intervals of $L$. 
Thus, the set $P$ of the vertex sets of the equivalence classes of $\sim$ is an interval partition of $L$. 
We consider for $l$ the quotient $L/P$ of $L$ by $P$ defined on $P$ in the following manner. 
Given distinct $I,J\in P$, $I<J\!\mod\! L/P$ if $i<j\!\mod\! L$ for $i\in I$ and $j\in J$. 
It is easy to verify that $L/P$ is a linear order. 
Lastly, since $L$ is discrete, $L[I]$ is isomorphic to 
$\omega$, $\omega^\star$ or $\omega^\star+\omega$ for each $I\in P$. 
\end{proof}
\end{rem}

Now, we examine Theorem~\ref{thm1_half_graph} in the finite case. 
Given $n\geq 1$, we consider the graph $H_{2n}$ defined on $\{0,\ldots,2n-1\}$ by 
$$E(H_{2n})=\bigcup_{0\leq p\leq n-1}\{\{2p,2q+1\}:p\leq q\leq n-1\}\ \text{(see Figure~\ref{H_2n})}.$$ 
Clearly, the cardinality of a finite half graph is even. 
Up to isomorphism, $H_{2n}$ is the unique finite half graph defined on $2n$ vertices. 

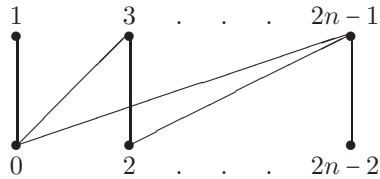
\begin{figure}[h]
\begin{center}
\setlength{\unitlength}{1cm}
\begin{picture}(11,2.5)
\put(3,0){0}\put(3,0.3){$\bullet$}

\put(3,2){1}\put(3,1.75){$\bullet$}

\put(4.5,2){3}\put(4.5,1.75){$\bullet$}
\put(4.5,0){2}\put(4.5,0.3){$\bullet$}
\put(5.2,0){.}
\put(5.8,0){.}
\put(6.4,0){.}
\put(5.2,2){.}
\put(5.8,2){.}
\put(6.4,2){.}
\put(7,0){$2n-2$}\put(7.45,0.3){$\bullet$}
\put(7,2){$2n-1$}\put(7.45,1.75){$\bullet$}

\put(3.1,0.4){\line(0,1){1.4}}
\put(3.1,0.4){\line(1,1){1.45}}
\put(3.1,0.4){\line(3,1){4.4}}

\put(4.6,0.4){\line(0,1){1.4}}
\put(4.6,0.4){\line(2,1){2.9}}

\put(7.55,0.4){\line(0,1){1.4}}

\end{picture}
\end{center}
\caption{The half graph $H_{2n}$.\label{H_2n}}
\end{figure}

\begin{prop}\label{prop1_half_graph}
For a finite and bipartite graph $\Gamma$, with $v(\Gamma)\geq 4$, the following assertions are equivalent 
\begin{enumerate}
\item $P_5\not\leq \Gamma$ and $\Gamma$ is prime;
\item $\Gamma$ is critical;
\item $\Gamma$ is a half graph. 
\end{enumerate}
\end{prop}

Proposition~\ref{prop1_half_graph} is an immediate consequence of the following two facts. 
The next fact is due to Boudabbous et al.~\cite{BDY18}. 

\begin{fac}\label{fac2_half_graph}
For a finite and prime graph $\Gamma$, $\Gamma$ is critical if and only if $\Gamma$ does not admit a prime induced subgraphs of size 5. 
\end{fac}

A simple characterization of finite and critical digraphs is provided in \cite{BI09} by using the primality graph (see Definition~\ref{defi_primality}). 
The next fact follows from it. 

\begin{fac}\label{fac3_half_graph}
Given a finite and bipartite graph $\Gamma$, with $v(\Gamma)\geq 4$, $\Gamma$ is critical if and only if 
$\Gamma$ is a half graph.  
\end{fac}

The next result is a consequence of Proposition~\ref{prop1_half_graph} and Theorem~\ref{thm2_pi_infinite}. 

\begin{cor}\label{cor1_infinite_half_graph}
A half graph $\Gamma$, with $v(\Gamma)\geq 4$, is prime. 
\end{cor}

\begin{proof}
There exists a bipartition $\{X,Y\}$ of $V(\Gamma)$, a linear order $L$ defined on $X$, and a bijection $\varphi$ from $X$ onto $Y$ such that 
$E(\Gamma)=\{\{x,\varphi(x')\}:x\leq x'\!\mod L\}$. 
By Proposition~\ref{prop1_half_graph}, we can suppose that $\Gamma$ is infinite. 
Consider a finite subset $F$ of $V(\Gamma)$. 
Let $X'$ be a finite subset of $X$ such that $F\cap X\subseteq X'$, 
$\varphi^{-1}(F\cap Y)\subseteq X'$, and 
$|X'|\geq 2$. 
Set $$F'=X'\cup\varphi(X').$$
Clearly $F\subseteq F'$. 
By considering $Y'=\varphi(X')$, the linear order $L'=L[X']$, and the bijection 
$\varphi_{\restriction X'}:X'\longrightarrow Y'$, we obtain that $\Gamma[F']$ is a 
half graph. 
By Proposition~\ref{prop1_half_graph}, $\Gamma[F'']$ is prime. 
To conclude, it suffices to use Theorem~\ref{thm2_pi_infinite}. 
\end{proof}

Now, we are ready to demonstrate Theorem~\ref{thm1_half_graph}. 

\begin{proof}[Proof of Theorem~\ref{thm1_half_graph}]
By Proposition~\ref{prop1_half_graph}, we can suppose that $\Gamma$ is infinite. 

To begin, suppose that $\Gamma$ is a discrete half graph. 
There exists a bipartition $\{X,Y\}$ of $V(\Gamma)$, a discrete linear order $L$ defined on $X$, and a bijection $\varphi$ from $X$ onto $Y$ such that 
$E(\Gamma)=\{\{x,\varphi(x')\}:x\leq x'\!\mod L\}$. 
By Corollary~\ref{cor1_infinite_half_graph}, 
$\Gamma$ is prime. 
Hence $\Gamma$ is connected. 
Since $\Gamma$ is a half graph, $K_2\oplus K_2\not\leq \Gamma$. 
It follows from Lemma~\ref{fac1_half_graph} that $P_5\not\leq \Gamma$. 
We verify that 
\begin{equation}\label{E1_prop2_half_graph}
\text{for every $x\in X$, $\Gamma-x$ is not prime.}
\end{equation}
First, suppose that $x$ is not the smallest element of $L$. 
Since $L$ is discrete, $x$ admits a predecessor $x^-$. 
It is easy to verify that $\{\varphi(x^-),\varphi(x)\}$ is a module of $\Gamma-x$. 
Second, suppose that $x$ is the smallest element of $\Gamma$. 
Clearly, $\varphi(x)$ is an isolated vertex of $\Gamma-x$, so $\Gamma-x$ is not prime. 
Thus \eqref{E1_prop2_half_graph} holds. 
Similarly, it follows from Remark~\ref{rem_infinite_half_graph} that 
$\Gamma-y$ is not prime for each $y\in Y$. 
Consequently $\Gamma$ is critical.

Conversely, suppose that $P_5\not\leq \Gamma$ and $\Gamma$ is critical. 
Since $\Gamma$ is bipartite, 
there exists a bipartition $\{X,Y\}$ of $V(\Gamma)$ such that $X$ and $Y$ are stable sets of $\Gamma$. 
To complete the proof, we establish the next claims. 
\end{proof}

\begin{defi}\label{defi_L_half_graph}
Since $\Gamma$ is prime, we have $N_\Gamma(x)\neq N_\Gamma(x')$ 
for distinct $x,x'\in X$. 
Moreover, since $P_5\not\leq \Gamma$, $K_2\oplus K_2\not\leq \Gamma$ by Lemma~\ref{fac1_half_graph}. 
It follows that for distinct $x,x'\in X$, we have 
$N_\Gamma(x)\subsetneq N_\Gamma(x')$ or $N_\Gamma(x')\subsetneq N_\Gamma(x)$. 
Therefore, we can define on $X$ a linear order $L$ as follows. 
Given distinct $x,x'\in X$, 
$$\text{$x<x'\hspace{-2mm}\mod L$ if $N_\Gamma(x)\supsetneq N_\Gamma(x')$.}$$ 
\end{defi}

We show that $\Gamma$ is the half graph defined from the linear order $L$ 
(see Claim~\ref{cla7_half_graph}). 
We have also to define a suitable bijection from $X$ onto $Y$ 
(see Definition~\ref{defi_varphi_half_graph}). 
We use the fact that $\Gamma$ is critical.

\begin{cla}\label{cla1_half_graph}
Given $x\in X$, if $\Gamma-x$ is disconnected, then the following assertions hold 
\begin{enumerate}
\item $\Gamma-x$ admits a unique isolated vertex $i_x$, and $i_x\in Y$;
\item $N_\Gamma(x)=Y$, so $x$ is the smallest element of $L$;
\item $i_x$ is the unique element of $V(\Gamma)\setminus\{x\}$ such that 
$\Gamma-\{x,i_x\}$ is prime. 
\end{enumerate}
\end{cla}

\begin{proof}
Since $\Gamma$ is connected, the set of the isolated vertices of $\Gamma-x$ is a module of $\Gamma$. 
Thus $|\{C\in\mathcal{C}(\Gamma-x):v(C)=1\}|\leq 1$. 
Furthermore, since $K_2\oplus K_2\leq \Gamma$, 
if $\Gamma-x$ admits at most one nontrivial component. 
Therefore $|\{C\in\mathcal{C}(\Gamma-x):v(C)\geq 2\}|\leq 1$. 
It follows that $\Gamma-x$ admits a unique isolated vertex $i_x$, and 
$\Gamma-\{x,i_x\}$ is connected. 
Since $i_x$ is an isolated vertex of $\Gamma-x$, $\{x,i_x\}\in E(\Gamma)$ because $\Gamma$ is connected. 
Hence $i_x\in Y$. 

Now, we verify that $N_\Gamma(x)=Y$. 
Let $y\in Y\setminus\{i_x\}$. 
Since $\Gamma-\{x,i_x\}$ is connected, there exits $x'\in X\setminus\{x\}$ such that 
$\{x',y\}\in E(\Gamma)$. 
Since $\Gamma[\{x,x',y,i_x\}]\not\simeq K_2\oplus K_2$, we obtain 
$\{x,y\}\in E(\Gamma)$. 
It follows that $N_\Gamma(x)=Y$. 
Hence $x$ is the smallest element of $L$. 

Lastly, we verify that $\Gamma-\{x,i_x\}$ is prime. 
Otherwise, $\Gamma-\{x,i_x\}$ admits a nontrivial module $M$. 
Since $\Gamma-\{x,i_x\}$ is connected and bipartite with bipartition 
$\{X\setminus\{x\},Y\setminus\{i_x\}\}$, we have $M\subseteq X\setminus\{x\}$ or 
$M\subseteq Y\setminus\{i_x\}$. 
Since $N_\Gamma(x)=Y$ and $N_\Gamma(i_x)=\{x\}$, $M$ is a module of $\Gamma$, which contradicts the fact that $\Gamma$ is prime. 
Consequently $\Gamma-\{x,i_x\}$ is prime. 
Moreover, consider $v\in V(\Gamma)\setminus\{x,i_x\}$. 
Since $i_x$ is isolated in $\Gamma-x$, it is also isolated in $\Gamma-\{x,v\}$. 
Therefore $\Gamma-\{x,v\}$ is not prime. 
It follows that $i_x$ is the unique element of $V(\Gamma)\setminus\{x\}$ such that 
$\Gamma-\{x,i_x\}$ is prime. 
\end{proof}

\begin{cla}\label{cla2_half_graph}
Let $x\in X$ such that $\Gamma-x$ is connected. 
For any nontrivial module $M$ of $\Gamma-x$, there exist 
$x^-,x^+\in Y$ such that 
$M=\{x^-,x^+\}$, $\{x,x^-\}\not\in E(\Gamma)$, and $\{x,x^+\}\in E(\Gamma)$. 
\end{cla}

\begin{proof}
Let $M$ be a nontrivial module of $\Gamma-x$. 
Since $\Gamma-x$ is connected, we have $M\subseteq X\setminus\{x\}$ or 
$M\subseteq Y$. 
In the first instance, $M$ is a module of $\Gamma$. 
Therefore $M\subseteq Y$. 
Set $M^-=\{y\in M:\{x,y\}\not\in E(\Gamma)\}$ and 
$M^+=\{y\in M:\{x,y\}\in E(\Gamma)\}$. 
Clearly, $M^-$ and $M^+$ are modules of $\Gamma$. 
Since $\Gamma$ is prime and $|M|\geq 2$, we obtain $|M^-|=1$ and $|M^+|=1$. 
Denote by $x^-$ the unique element of $M^-$, and denote by 
$x^+$ the unique element of $M^+$. 
We obtain $M=\{x^-,x^+\}$. 
Furthermore, $\{x,x^-\}\not\in E(\Gamma)$ and $\{x,x^+\}\in E(\Gamma)$. 
\end{proof}

\begin{cla}\label{cla3_half_graph}
Given $x\in X$, if $\Gamma-x$ is connected, then there exist $x^-,x^+\in Y$ satisfying the following assertions 
\begin{enumerate}
\item $\{x^-,x^+\}$ is the only nontrivial module of $\Gamma-x$;
\item $\{x,x^-\}\not\in E(\Gamma)$ and $\{x,x^+\}\in E(\Gamma)$;
\item for every $u\in X$, if $u<x\hspace{-1mm}\mod L$, then $\{u,x^-\}\in E(\Gamma)$;
\item for every $u\in X$, if $x<u\hspace{-1mm}\mod L$, then $\{u,x^+\}\not\in E(\Gamma)$;
\item $\Gamma-\{x,x^-\}$ and $\Gamma-\{x,x^+\}$ are prime;
\item $x^+$ is the unique element of $V(\Gamma)\setminus\{x\}$ such that 
$\{x,x^+\}\in E(\Gamma)$ and $\Gamma-\{x,x^+\}$ is prime. 
\end{enumerate}
\end{cla}

\begin{proof}
Since $\Gamma$ is critical, $\Gamma-x$ admits a nontrivial module $M$. 
By Claim~\ref{cla2_half_graph}, 
there exist $x^-,x^+\in Y$ such that $M=\{x^-,x^+\}$, $\{x,x^-\}\not\in E(\Gamma)$, and 
$\{x,x^+\}\in E(\Gamma)$. 
Hence $\{x^-,x^+\}$ is a nontrivial module of $\Gamma-x$. 

For a contradiction, suppose that $M$ is not the only nontrivial module of $\Gamma-x$. 
Thus, there exists a nontrivial module $N$ of $\Gamma-x$ such that $N\neq M$. 
By Claim~\ref{cla2_half_graph}, 
there exist $z^-,z^+\in Y$ such that $N=\{z^-,z^+\}$, $\{x,z^-\}\not\in E(\Gamma)$, and 
$\{x,z^+\}\in E(\Gamma)$. 
If $M\cap N\neq\emptyset$, then $M\cup N$ is a nontrivial module of $\Gamma-x$ of size 3, which contradicts Claim~\ref{cla2_half_graph}. 
Hence $M\cap N=\emptyset$. 
We show that $M\cup N$ is a module of $\Gamma-x$. 
Let $u\in (X\setminus\{x\})\setminus(M\cup N)$. 
It suffices to verify that $M\cup N$ is a module of $\Gamma[M\cup N\cup\{u\}]$. 
Suppose that there exists $v\in M\cup N$ such that $\{u,v\}\in E(\Gamma)$. 
For instance, suppose that $v\in M$. 
Since $M$ is a module of $\Gamma-x$, we have $\{u,x^-\},\{u,x^+\}\in E(\Gamma)$. 
We have $\{u,x^-\}\in E(\Gamma)$, $\{x,x^-\}\not\in E(\Gamma)$, and 
$\{x,z^+\}\in E(\Gamma)$. 
Since $K_2\oplus K_2\not\leq \Gamma$, we obtain $\{u,z^+\}\in E(\Gamma)$. 
Since $\{z^-,z^+\}$ is a module of $\Gamma-x$, we have $\{u,z^-\}\in E(\Gamma)$. 
Therefore, $\{u,w\}\in E(\Gamma)$ for every $w\in M\cup N$. 
It follows that $M\cup N$ is a module of $\Gamma-x$, which contradicts 
Claim~\ref{cla2_half_graph} because $|M\cup N|=4$. 
Consequently, 
$\{x^-,x^+\}$ is the only nontrivial module of $\Gamma-x$. 
It follows that $\Gamma-\{x,x^-\}$ and $\Gamma-\{x,x^+\}$ are prime. 
 
 Let $u\in X$ such that $u<x\hspace{-1mm}\mod L$. 
 Since $u<x\hspace{-1mm}\mod L$, we have 
 $N_\Gamma(u)\supseteq N_\Gamma(x)$. 
 Hence $\{u,x^+\}\in E(\Gamma)$ because $\{x,x^+\}\in E(\Gamma)$. 
 Since $\{x^-,x^+\}$ is a module of $\Gamma-x$, we obtain $\{u,x^-\}\in E(\Gamma)$. 
 
 Let $u\in X$ such that $x<u\hspace{-1mm}\mod L$. 
 Since $x<u\hspace{-1mm}\mod L$, we have $N_\Gamma(x)\supseteq N_\Gamma(u)$. 
 Hence $\{u,x^-\}\not\in E(\Gamma)$ because $\{x,x^-\}\not\in E(\Gamma)$. 
 Since $\{x^-,x^+\}$ is a module of $\Gamma-x$, we obtain 
 $\{u,x^+\}\not\in E(\Gamma)$. 
 
 As previously seen, $\Gamma-\{x,x^-\}$ and $\Gamma-\{x,x^+\}$ are prime. 
 Now, consider $v\in V(\Gamma)\setminus\{x,x^-,x^+\}$. 
 Clearly, $\{x^-,x^+\}$ is a nontrivial module of $\Gamma-\{x,v\}$, so 
 $\Gamma-\{x,v\}$ is not prime. 
 Since $\{x,x^-\}\not\in E(\Gamma)$, 
 $x^+$ is the unique element of $V(\Gamma)\setminus\{x\}$ such that 
$\{x,x^+\}\in E(\Gamma)$ and $\Gamma-\{x,x^+\}$ is prime. 
\end{proof}

\begin{defi}\label{defi_varphi_half_graph}
We define a function $\varphi:X\longrightarrow Y$ as follows. 
Given $x\in X$, 
\begin{equation*}
\varphi(x)=\ 
\begin{cases}
\text{$i_x$ if $\Gamma-x$ is disconnected (see Claim~\ref{cla1_half_graph}),}\\
\text{or}\\
\text{$x^+$ if $\Gamma-x$ is connected (see Claim~\ref{cla3_half_graph}).}
\end{cases}
\end{equation*}
\end{defi}

The next claim follows easily from Claim~\ref{cla1_half_graph} and \ref{cla3_half_graph}.

\begin{cla}\label{cla4_half_graph}
For every $x\in X$, 
$\varphi(x)$ is the unique element of $V(\Gamma)\setminus\{x\}$ such that 
$\{x,\varphi(x)\}\in E(\Gamma)$ and $\Gamma-\{x,\varphi(x)\}$ is prime. 
\end{cla}

In the next two claims, we verify that $\varphi$ is bijective. 

\begin{cla}\label{cla5_half_graph}
$\varphi$ is injective. 
\end{cla}

\begin{proof}
Consider distinct $u,v\in X$. 
For instance, suppose that $u<v \hspace{-1mm}\mod L$. 
In particular, $v$ is not the smallest element of $L$. 
It follows from Claim~\ref{cla1_half_graph} that $\Gamma-v$ is connected. 
By Claim~\ref{cla3_half_graph}, 
there exist $v^-,v^+\in Y$ such that $\{v,v^-\}\not\in E(\Gamma)$, 
$\{v,v^+\}\in E(\Gamma)$, and 
$\{v^-,v^+\}$ is the only nontrivial module of $\Gamma-v$. 
We have $\varphi(v)=v^+$. 

First, suppose that $\Gamma-u$ is disconnected. 
We have $\varphi(u)=i_u$, where $i_u$ is the unique isolated vertex of $\Gamma-u$ by Claim~\ref{cla1_half_graph}. 
We obtain 
$\{v,\varphi(u)\}\not\in E(\Gamma)$. 
Thus $\varphi(u)\neq\varphi(v)$ because $\{v,\varphi(v)\}\in E(\Gamma)$ (see Claim~\ref{cla4_half_graph}). 

Second, suppose that $\Gamma-u$ is connected. 
By Claim~\ref{cla3_half_graph}, 
there exist $u^-,u^+\in Y$ such that $\{u,u^-\}\not\in E(\Gamma)$, 
$\{u,u^+\}\in E(\Gamma)$, and 
$\{u^-,u^+\}$ is the only nontrivial module of $\Gamma-u$. 
We have $\varphi(u)=u^+$. 
Since $u<v \hspace{-1mm}\mod L$, it follows from the fourth assertion of Claim~\ref{cla3_half_graph} applied to $u$ that $\{v,\varphi(u)\}\not\in E(\Gamma)$. 
Since $\{v,\varphi(v)\}\in E(\Gamma)$ (see Claim~\ref{cla4_half_graph}), 
$\varphi(u)\neq\varphi(v)$. 
\end{proof}

\begin{cla}\label{cla6_half_graph}
$\varphi$ is surjective. 
\end{cla}

\begin{proof}
Let $v\in Y$. 
Since $\Gamma$ is critical, $\Gamma-v$ is not prime. 

First, suppose that $\Gamma-v$ is disconnected. 
As in Claim~\ref{cla1_half_graph}, we obtain that $\Gamma-v$ admits an isolated vertex $i_v$. 
Thus $N_\Gamma(i_v)=\{v\}$. 
Since $\{i_v,\varphi(i_v)\}\in E(\Gamma)$, we obtain $\varphi(i_v)=v$. 

Second, suppose that $\Gamma-v$ is connected. 
As in Claim~\ref{cla3_half_graph}, there exist $v^-,v^+\in X$ such that 
$\{v^-,v^+\}$ is the only nontrivial module of $\Gamma-v$, $\{v,v^-\}\not\in E(\Gamma)$, and 
$\{v,v^+\}\in E(\Gamma)$. 
Furthermore, $\Gamma-\{v,v^-\}$ and $\Gamma-\{v,v^+\}$ are prime. 
Thus $\Gamma-\{v,v^+\}$ is prime, and $\{v,v^+\}\in E(\Gamma)$. 
It follows from Claim~\ref{cla4_half_graph} that $v=\varphi(v^+)$. 
\end{proof}

It follows from Claims~\ref{cla5_half_graph} and \ref{cla6_half_graph} that $\varphi$ is bijective. 

\begin{cla}\label{cla7_half_graph}
$\Gamma$ is the half graph defined from the linear order $L$, and the bijection $\varphi$. 
\end{cla}

\begin{proof}
Consider distinct $u,x\in X$. 
We have to verify that 
$$\text{$\{u,\varphi(x)\}\in E(\Gamma)$ if and only if $u\leq x\hspace{-2mm}\mod L$.}$$
Suppose that $u\leq x\hspace{-1mm}\mod L$. 
We obtain $N_\Gamma(x)\subseteq N_\Gamma(u)$. 
By Claim~\ref{cla4_half_graph}, $\varphi(x)\in N_\Gamma(x)$. 
Hence $\varphi(x)\in N_\Gamma(u)$. 
Conversely, suppose that $x<u\hspace{-1mm}\mod L$. 
In particular, $u$ is not the smallest element of $L$. 
It follows from Claim~\ref{cla1_half_graph} that $\Gamma-u$ is connected. 
By the fourth assertion of Claim~\ref{cla3_half_graph} applied to $x$, $\{u,x^+\}\not\in E(\Gamma)$, 
that is, 
$\{u,\varphi(x)\}\not\in E(\Gamma)$. 
\end{proof}

\begin{cla}\label{cla8_half_graph}
Given $x\in X$, if $x$ is not the smallest element of $L$, then $x$ admits a predecessor in $L$. 
\end{cla}

\begin{proof}
Let $x\in X$. 
Suppose that $x$ is not the smallest element of $L$. 
It follows from Claim~\ref{cla1_half_graph} that $\Gamma-x$ is connected. 
By Claim~\ref{cla3_half_graph}, there exist $x^-,x^+\in Y$ such that 
$\{x^-,x^+\}$ is the only nontrivial module of $\Gamma-x$, $\{x,x^-\}\not\in E(\Gamma)$, and 
$\{x,x^+\}\in E(\Gamma)$. 
Furthermore, for every $u\in X$, 
we have 
\begin{equation}\label{E6_prop2_half_graph}
\text{if $u<x\hspace{-2mm}\mod L$, then 
$\{u,x^-\}\in E(\Gamma)$,}
\end{equation}
by the third assertion of Claim~\ref{cla3_half_graph} applied to $x$. 
Set $$t=\varphi^{-1}(x^-).$$ 
By Claim~\ref{cla4_half_graph}, $\{t,\varphi(t)\}\in E(\Gamma)$, that is, 
$\{t,x^-\}\in E(\Gamma)$. 
We obtain $x^-\in N_\Gamma(t)\setminus N_\Gamma(x)$. 
Hence $N_\Gamma(t)\supsetneq N_\Gamma(x)$, so $t<x\hspace{-1mm}\mod L$. 
We prove that $t$ is the predecessor of $x$. 
It suffices to verify that $$(t,x)_L=\emptyset.$$
First, suppose that $\Gamma-t$ is disconnected. 
By Claim~\ref{cla1_half_graph}, there exists $i_t\in Y$ such that $i_t$ is an isolated vertex of $\Gamma-t$. 
Since $\varphi(t)=i_t$, $i_t=x^-$. 
We obtain that $\{u,x^-\}\not\in E(\Gamma)$ for every $u\in V(\Gamma)\setminus\{t,x^-\}$. 
It follows from \eqref{E6_prop2_half_graph} that $(t,x)_L=\emptyset$. 
Second, suppose that $\Gamma-t$ is connected. 
By Claim~\ref{cla3_half_graph}, there exist $t^-,t^+\in Y$ such that 
$\{t^-,t^+\}$ is the only nontrivial module of $\Gamma-t$, $\{t,t^-\}\not\in E(\Gamma)$, and 
$\{t,t^+\}\in E(\Gamma)$. 
Furthermore, for every $u\in X$ such that $t<u\hspace{-1mm}\mod L$, we have 
$\{u,t^+\}\not\in E(\Gamma)$ by the fourth assertion of Claim~\ref{cla3_half_graph} applied to $t$. 
Recall that $t^+=\varphi(t)$. 
Since $t=\varphi^{-1}(x^-)$, we obtain $t^+=x^-$. 
Therefore, for every $u\in X$ such that $t<u\hspace{-1mm}\mod L$, we have 
$\{u,x^-\}\not\in E(\Gamma)$. 
It follows from \eqref{E6_prop2_half_graph} that $(t,x)_L=\emptyset$. 
\end{proof}

By Remark~\ref{rem_infinite_half_graph}, 
$\Gamma$ is also the half graph defined from the linear order $\varphi(L)^\star$ defined on $Y$, and the bijection $\varphi^{-1}:Y\longrightarrow X$. 
The analogue of Claim~\ref{cla8_half_graph} for $\varphi(L)^\star$ follows. 

\begin{cla}\label{cla9_half_graph}
Given $y\in Y$, if $y$ is not the smallest element of $\varphi(L)^\star$, then $y$ admits a predecessor in $\varphi(L)^\star$. 
\end{cla}

The next claim is an immediate consequence of Claims~\ref{cla9_half_graph}. 

\begin{cla}\label{cla11_half_graph}
Given $x\in X$, if $x$ is not the largest element of $L$, then $x$ admits a successor in $L$. 
\end{cla}

It follows from Claims~\ref{cla8_half_graph} and \ref{cla11_half_graph} that $L$ is discrete, which completes the proof of Theorem~\ref{thm1_half_graph}. 

As announced in Subsection~\ref{Main results}, we discuss Theorem~\ref{cor1_thm_main_1} by using Theorems~\ref{thm_main_2} and 
\ref{thm1_half_graph}. 

\begin{rem}\label{rem_thm_24}
We denote by $\mathbb{Q}$ the set of rational numbers, and $L_\mathbb{Q}$ denotes the usual linear order on $\mathbb{Q}$. 
Obviously, $L_\mathbb{Q}$ is not discrete. 
We consider  the graph $G$ defined on 
$\{0,1,2,3\}\cup(\{0,1\}\times\mathbb{Q})$ by 
\begin{align*}
E(G)=\{\{0,1\},\{1,2\},\{2,3\}\}&\cup\{\{1,(1,q)\}:q\in\mathbb{Q}\}\\&\cup(\bigcup_{q\in\mathbb{Q}}\{\{(0,q),(1,r)\}:r\geq q\}).
\end{align*}
Set $X=\{0,1,2,3\}$, $Y=\{0\}\times\mathbb{Q}$ and $Z=\{1\}\times\mathbb{Q}$. 
We have $G[X]$ is prime because $G[X]=P_4$. 
We consider the 2-structure $\sigma_G$ associated with $G$. 
Since $G[X]$ is prime, $\sigma_G[X]$ is prime too. 
We have $Y=\langle X\rangle_{\sigma_G}$, $Z=X_{\sigma_G}(0)$, and 
$p_{(\sigma_G,\overline{X})}=\{Y,Z\}$. 
Furthermore, it is not difficult to verify that 
\begin{equation}\label{E2_rem_thm_24}
\Gamma_{(\sigma_G,\overline{X})}=G[Y\cup Z].
\end{equation}
We verify that $\sigma_G$ is finitely $\overline{X}$-critical (see Definition~\ref{defi_finitely}), without being 
$\overline{X}\!$-critical. 

We show that Statement~(Sk) holds for every odd integer $k\geq 1$. 
Let $W$ be a finite and nonempty subset of $Y\cup Z$ such that 
$W\in\varepsilon_{(\sigma_G,\overline{X})}$ (see Notation~\ref{nota_outsideg}). 
We have to show that $W$ is even. 
If $W\cap Y=\emptyset$, then $\{0\}\cup W$ is a module of $\sigma_G[X\cup W]$ because 
$Z=X_{\sigma_G}(0)$. 
Hence $W\cap Y\neq\emptyset$. 
We denote the elements of $W\cap Y$ by $(0,q_0),\ldots,(0,q_m)$, where $m\geq 0$, in such a way that 
$q_0<\cdots <q_m$, when $m\geq 1$. 
Set $Z^-=\{j<q_0:(1,j)\in W\}$. 
Since $Z=X_{\sigma_G}(0)$, $\{0\}\cup(\{1\}\times Z^-)$ is a module of $\sigma_G[X\cup W]$. 
Hence $Z^-=\emptyset$. 
Set $Z^+=\{j\geq q_m:(1,j)\in W\}$. 
We obtain that $\{1\}\times Z^+$ is a module of $\sigma_G[X\cup W]$. 
Hence $|Z^+|\leq 1$. 
If $Z^+=\emptyset$, then $(X\cup W)\setminus\{(0,q_m)\}$ 
is a module of $\sigma_G[X\cup W]$ because $(0,q_m)\in\langle X\rangle_{\sigma_G}$. 
Thus $|Z^+|=1$. 
Therefore, $|W|=2$ if $m=0$. 
Suppose that $m\geq 1$, and set $Z_i=\{j\in[q_i,q_{i+1}):(1,j)\in W\}$ for $i=0,\ldots,m-1$. 
Given $i=0,\ldots,m-1$, we have $\{1\}\times Z_i$ is a module of $\sigma_G[X\cup W]$. 
Hence $|Z_i|\leq 1$. 
Moreover, $\{(0,q_i),(0,q_{i+1})\}$ is a module of $\sigma_G[X\cup W]$ if $Z_i=\emptyset$. 
Therefore, $|Z_i|=1$. 
Consequently, $Z^-=\emptyset$, $|Z^+|=1$, and $|Z_i|=1$ for $i=0,\ldots,m-1$. 
Thus, $|W\cap Z|=m+1$, and hence $|W|=2m+2$. 

We prove that $\sigma_G$ is finitely $\overline{X}$-critical. 
Let $F$ be a finite subset of $Y\cup Z$. 
There exists a finite subset $F'$ of $\mathbb{Q}$ such that $|F'|\geq 2$ and 
$F\subseteq(\{0,1\}\times F')$. 
We have $G[\{0,1\}\times F']\simeq H_{2\times |F'|}$ (see Figure~\ref{H_2n}). 
It follows from Proposition~\ref{prop1_half_graph} that $G[\{0,1\}\times F']$ is critical. 
Set $\widetilde{F}=\{0,1\}\times F'$. 
We obtain that 
\begin{equation}\label{E1_rem_thm_24}
\text{$F\subseteq\widetilde{F}$ and $G[\widetilde{F}]$ is critical.}
\end{equation}
It follows from \eqref{E2_rem_thm_24} and \eqref{E1_rem_thm_24} that 
$\Gamma_{(\sigma_G[X\cup\widetilde{F}],\widetilde{F})}$ is critical. 
Since Statement (S5) holds, it follows from Theorem~\ref{thm_main_2} that $\sigma_G[X\cup\widetilde{F}]$ is $\widetilde{F}$\!-critical. 
Consequently, 
$\sigma_G$ is finitely $\overline{X}$-critical. 

Since $\sigma_G$ is finitely $\overline{X}$-critical, it follows from Theorem~\ref{cor1_thm_main_1} that 
$\sigma_G$ is prime. 
Lastly, we verify that $\sigma_G$ is not 
$\overline{X}\!$-critical. 
To begin, we verify that $G[Y\cup Z]$ is a non discrete half graph. 
Clearly, $G[Y\cup Z]$ is bipartite with bipartition $\{Y,Z\}$. 
Consider the bijection $\varphi:Y\longrightarrow Z$, which maps $(0,q)$ to $(1,q)$ for each 
$q\in\mathbb{Q}$. 
Moreover, consider the linear order $L_Y$ defined on $Y$ as follows. 
Given distinct $q,r\in\mathbb{Q}$, $(0,q)<(0,r)\hspace{-1mm}\mod L_Y$ if 
$q<r\hspace{-1mm}\mod L_\mathbb{Q}$. 
Clearly, $G[Y\cup Z]$ is the half graph defined from $L_Y$ and $\varphi$. 
Since $L_Y\simeq L_\mathbb{Q}$, $G[Y\cup Z]$ is not discrete. 

Since Statement (S5) holds,  $P_5\not\leq\Gamma_{(\sigma_G,\overline{X})}$ by Lemma~\ref{lem1_component}. 
Since $G[Y\cup Z]$ is a non discrete half graph, 
$\Gamma_{(\sigma_G,\overline{X})}$ is a non discrete half graph by \eqref{E2_rem_thm_24}. 
It follows from Theorem~\ref{thm1_half_graph} that $\Gamma_{(\sigma_G,\overline{X})}$ is not critical. 
Clearly, $G[Y\cup Z]$ is connected. 
Therefore, $\Gamma_{(\sigma_G,\overline{X})}$ is connected by \eqref{E2_rem_thm_24}. 
It follows from Theorem~\ref{thm_main_2} that $\sigma_G$ is not 
$\overline{X}\!$-critical. 
Since $\sigma_G$ is prime, there exists $v\in\overline{X}$ such that $\sigma_G-v$ is prime. 
In fact, we have $\sigma_G-w$ is prime for every $w\in\overline{X}$. 
\end{rem}

\appendix

\section*{Appendices}

\section{Description of partially critical 2-structures}\label{A_representation}

We use the following notation. 

\begin{nota}\label{prop_component_bis}
Given a 2-structure $\sigma$, consider $X\subsetneq V(\sigma)$ such that $\sigma[X]$ is prime. 
Suppose that Statement (S3) holds. 
Let $C$ be a component of $\Gamma_{(\sigma,\overline{X})}$. 
Consider $x,x'\in B_q^C$ and $y,y'\in D_q^C$ 
(see Notation~\ref{nota_prop_component}) such that 
$\{x,y\}\in E(\Gamma_{(\sigma,\overline{X})})$ and 
$\{x',y'\}\in E(\Gamma_{(\sigma,\overline{X})})$. 
Since $C$ is connected, it follows from Fact~\ref{fac1_first_results} that 
$[x,y]_\sigma=[x',y']_\sigma$. 
We denote $[x,y]_\sigma$ by $s_C$. 
\end{nota}

\begin{fac}\label{fac_description}
Given a 2-structure $\sigma$, consider $X\subsetneq V(\sigma)$ such that 
$\sigma[X]$ is prime. 
Suppose that Statement (S5) holds, and $\sigma$ is $\overline{X}\!$-critical. 
Under these assumptions, $\sigma$ is entirely determined by 
$\sigma[X]$, $q_{(\sigma,\overline{X})}$, 
$\Gamma_{(\sigma,\overline{X})}$, and 
$\{s_C:C\in\mathcal{C}(\Gamma_{(\sigma,\overline{X})})\}$. 
\end{fac}

\begin{proof} 
We make the following preliminary observation. 
Since Statement (S5) holds, Statement (S3) holds as well 
(see Remark~\ref{rem1_conditions_Ck}). 
Since $\sigma$ is prime, it follows from Corollary~\ref{cor1_first_results} that 
\begin{equation}\label{E1_fac_description}
\text{$\Gamma_{(\sigma,\overline{X})}$ has no isolated vertices.}
\end{equation}

We have to determine $[x,y]_\sigma$ for distinct vertices $x,y$ of $\sigma$ such that 
$\{x,y\}\setminus X\neq\emptyset$. 
To begin, consider $x\in X$ and $y\in\overline{X}$. 
Since ${\rm Ext}_\sigma(X)=\emptyset$, $[x,y]_\sigma$ is determined by the block of $q_{(\sigma,\overline{X})}$ containing $y$. 
For instance, if $y\in X_\sigma^{(e,f)}(\alpha)$, where $e,f\in E(\sigma)$ and $\alpha\in X$, we have 
\begin{equation*}
[x,y]_\sigma=\ 
\begin{cases}
\text{$(f,e)$ if $x=\alpha$}\\
\text{or}\\
\text{$[x,\alpha]_\sigma$ if $x\neq\alpha$.}
\end{cases}
\end{equation*}

Now, we consider distinct $x,y\in\overline{X}$. 
To begin, we suppose that $x$ and $y$ belong to the same block of 
$p_{(\sigma,\overline{X})}$. 

\begin{itemize}
\item First, suppose that $\{x,y\}\cap B_q^s\neq\emptyset$, where 
$B_q^s\in q_{(\sigma,\overline{X})}^s$. 
Since $B_q^s\in q_{(\sigma,\overline{X})}$, there exists $e\in E(\sigma)$ such that $B_q^s=\langle X\rangle_\sigma^{(e,e)}$ or 
$B_q^s=X_\sigma^{(e,e)}(\alpha)$, where $\alpha\in X$. 
Since $\Gamma_{(\sigma,\overline{X})}$ has no isolated vertices 
(see \eqref{E1_fac_description}), it follows from Lemma~\ref{l1_p_and_q} 
 that $B_q^s\in p_{(\sigma,\overline{X})}$. 
 Hence $x,y\in B_q^s$. 
Since $\sigma$ is prime, it follows from the first assertion of Lemma~\ref{l3_p_and_q} that 
$[x,y]_\sigma=(e,e)$. 
\item Second, suppose that $x\in B_q^a$ and $y\in D_q^a$, where 
$B_q^a$ and $D_q^a$ are distinct elements of $q_{(\sigma,\overline{X})}^a$. 
Recall that $\Gamma_{(\sigma,\overline{X})}$ has no isolated vertices 
(see \eqref{E1_fac_description}). 
Therefore, we can apply Lemmas~\ref{l1_p_and_q} and \ref{l2_p_and_q} as follows. 
Since $x$ and $y$ belong to the same block of 
$p_{(\sigma,\overline{X})}$, 
it follows from Lemma~\ref{l1_p_and_q} that 
$B_q^a\cup D_q^a\in p_{(\sigma,\overline{X})}$. 
We use Lemma~\ref{l2_p_and_q} to determine 
$[x,y]_\sigma$. 
For instance, if $x\in\langle X\rangle_\sigma^{(e,f)}$ and 
$y\in\langle X\rangle_\sigma^{(f,e)}$, where 
$e,f\in E(\sigma)$ with $e\neq f$, then $[x,y]_\sigma=(e,f)$ by the first assertion of 
Lemma~\ref{l2_p_and_q}. 
\item Third, suppose that $x,y\in B_q^a$, where $B_q^a\in q_{(\sigma,\overline{X})}^a$. 
Since $B_q^a\in q_{(\sigma,\overline{X})}^a$, there exist distinct $e,f\in E(\sigma)$ such that 
\begin{equation*}
B_q^a=\ 
\begin{cases}
\langle X\rangle_\sigma^{(e,f)}\\
\text{or}\\
\text{$X_\sigma^{(e,f)}(\alpha)$, where $\alpha\in X$.}
\end{cases}
\end{equation*}
To determine $[x,y]_\sigma$, we describe $\sigma[B_q^a]$ in the following manner. 
Let $C$ be the component of $\Gamma_{(\sigma,\overline{X})}$ containing $x$. 
Since $\Gamma_{(\sigma,\overline{X})}$ has no isolated vertices 
(see \eqref{E1_fac_description}), it follows from the second assertion of Proposition~\ref{prop_component} that $B_q^a\subseteq V(C)$. 
For distinct $u,v\in B_q^a$, set 
$$\text{$u<v\!\mod l(B_q^a)$ if $[u,v]_\sigma=(e,f)$.}$$
Since $\sigma$ is prime, it follows from the second assertion of Lemma~\ref{l3_p_and_q} that $l(B_q^a)$ is a linear order. 
For instance, suppose that $B_q^a=B_q^C$ (see Notation~\ref{nota_prop_component}). 
Recall that $C$ is a bipartite graph, with bipartition $\{B_q^C,D_q^C\}$. 
Since $|B_q^a|\geq 2$, it follows from Theorem~\ref{thm_main_2} that $v(C)\geq 4$ and $C$ is critical. 
Moreover, $P_5\not\leq C$ by Lemma~\ref{lem1_component}. 
It follows from Theorem~\ref{thm1_half_graph} that 
$C$ is a discrete half graph. 
Precisely, for distinct $u,v\in B_q^a$, set 
$$\text{$u<v\!\mod L(B_q^a)$ if $N_C(u)\supsetneq N_C(v)$ 
(see Definition~\ref{defi_L_half_graph}).}$$
Furthermore, we define a function $\varphi(B_q^a):B_q^a\longrightarrow D_q^C$ as in Definition~\ref{defi_varphi_half_graph}. 
By Claims~\ref{cla5_half_graph} and \ref{cla6_half_graph}, $\varphi(B_q^a)$ is bijective. 
Lastly, by Claim~\ref{cla7_half_graph}, 
$C$ is the half graph defined from the linear order $L(B_q^a)$, and the bijection 
$\varphi(B_q^a)$. 
Consider distinct $u,v\in B_q^a$ such that $u<v\!\mod L(B_q^a)$. 
It follows that $\varphi(B_q^a)(u)\in N_C(u)\setminus N_C(v)$. 
By Corollary~\ref{cor0_first_results}, 
\begin{equation*}
[u,v]_\sigma=\ 
\begin{cases}
\text{$(f,e)$ if $B_q^a=\langle X\rangle_\sigma^{(e,f)}$}\\
\text{or}\\
\text{$(e,f)$ if $B_q^a=X_\sigma^{(e,f)}(\alpha)$.}
\end{cases}
\end{equation*}
Given distinct $u,v\in B_q^a$, it follows that 
\begin{equation*}
[u,v]_\sigma=(e,f)\ \text{if and only if}\ 
\begin{cases}
\text{$N_C(v)\supsetneq N_C(u)$ and $B_q^a=\langle X\rangle_\sigma^{(e,f)}$}\\
\text{or}\\
\text{$N_C(u)\supsetneq N_C(v)$ and $B_q^a=X_\sigma^{(e,f)}(\alpha)$.}
\end{cases}
\end{equation*}
Furthermore, observe that 
\begin{equation*}
l(B_q^a)=\ 
\begin{cases}
\text{$L(B_q^a)^\star$ if $B_q^a=\langle X\rangle_\sigma^{(e,f)}$}\\
\text{and}\\
\text{$L(B_q^a)$ if $B_q^a=X_\sigma^{(e,f)}(\alpha)$.}
\end{cases}
\end{equation*}
\end{itemize}
Lastly, we suppose that $x\in B_p$ and $y\in D_p$, where $B_p$ and $D_p$ are distinct elements of $p_{(\sigma,\overline{X})}$. 
\begin{itemize}
\item First, suppose that $\{x,y\}\not\in E(\Gamma_{(\sigma,\overline{X})})$. 
Suppose that $B_p=\langle X\rangle_\sigma$. 
There exist $e,f\in E(\sigma)$ such that $x\in \langle X\rangle_\sigma^{(e,f)}$. 
By the first assertion of Lemma~\ref{lem_EHR}, $[x,y]_\sigma=(e,f)$. 
Suppose that there exist distinct $\alpha,\beta\in X$ such that 
$B_p=X_\sigma(\alpha)$ and $B_q=X_\sigma(\beta)$. 
By the second assertion of Lemma~\ref{lem_EHR}, 
$[x,y]_\sigma=[\alpha,\beta]_\sigma$. 
\item Second, suppose that $\{x,y\}\in E(\Gamma_{(\sigma,\overline{X})})$. 
There exist $B_q,D_q\in q_{(\sigma,\overline{X})}$ such that $x\in B_q$ and $y\in D_q$. 
Thus $B_q\subseteq B_p$ and $D_q\subseteq D_p$. 
Denote by $C$ the component of $\Gamma_{(\sigma,\overline{X})}$ containing $x$ and $y$. 
We obtain $x\in V(C)\cap B_q$ and $y\in V(C)\cap D_q$. 
Therefore $x\in B_q^C$ and $y\in D_q^C$ (see Notation~\ref{nota_prop_component}). 
Hence 
$[x,y]_\sigma=s_C$ (see Notation~\ref{prop_component_bis}). \qedhere
\end{itemize}
\end{proof}

\begin{rem}\label{rem_description}
Given a 2-structure $\sigma$, consider $X\subsetneq V(\sigma)$ such that 
$\sigma[X]$ is prime. 
Suppose that Statement (S3) holds, and $\sigma$ is $\overline{X}\!$-critical. 
Let $C\in\mathcal{C}(\Gamma_{(\sigma,\overline{X})})$ such that $v(C)>2$. 
Since $C$ is prime  by Theorem~\ref{thm_main_1}, we have 
$[x,y]_\sigma\neq s_C$ for any $x\in B_q^C$ and $y\in D_q^C$ such that $\{x,y\}\not\in E(\Gamma_{(\sigma,\overline{X})})$. 
\end{rem}

We pursue by determining the modules created by partial criticality. 
We use the following notation. 

\begin{nota}\label{nota_varphi}
Given a 2-structure $\sigma$, consider $X\subsetneq V(\sigma)$ such that 
$\sigma[X]$ is prime. 
Suppose that Statement (S5) holds, and $\sigma$ is $\overline{X}\!$-critical. 
Consider a component $C$ of $\Gamma_{(\sigma,\overline{X})}$ such that 
$v(C)\geq 4$. 
By Theorem~\ref{thm_main_2}, $C$ is critical. 
By Lemma~\ref{lem1_component}, $P_5\not\leq C$. 
It follows from Theorem~\ref{thm1_half_graph} that 
$C$ is a half graph defined from a discrete linear order $L$ defined on $B_q^C$, and a bijection $\varphi$ from $B_q^C$ onto $D_q^C$. 

For distinct $u,v\in B_q^C$, we have 
$$\text{$\{u,\varphi(v)\}\in E(C)$ if and only if $u\leq v\hspace{-2mm}\mod L$.}$$
It follows that for distinct $u,v\in B_q^C$, 
$$\text{$u< v\hspace{-2mm}\mod L$ if and only if $N_C(u)\supsetneq N_C(v)$.}$$
Thus, the linear order $L$ is unique, it is denoted by $L_C$. 

Now, consider $u\in B_q^C$. 
First, suppose that $u$ is the largest element of $L_C$. 
We obtain that $N_C(u)\subsetneq N_C(v)$ for each $v\in B_q^C\setminus\{u\}$. 
It follows that $N_C(u)$ is a module of $C$. 
Hence $|N_C(u)|=1$, and $\varphi(u)$ is the unique element of $N_C(u)$. 
Second, suppose that $u$ is not the largest element of $L_C$. 
Since $L_C$ is discrete, $u$ admits a successor $u^+$ in $L_C$. 
It follows that $N_C(u)\setminus N_C(u^+)$ is a module of $C$. 
Hence $|N_C(u)\setminus N_C(u^+)|=1$, and $\varphi(u)$ is the unique element of 
$N_C(u)\setminus N_C(u^+)$. 
Consequently, the bijection $\varphi$ is unique, it is denoted by $\varphi_C$. 

Lastly, suppose that $\Gamma_{(\sigma,\overline{X})}$ admits a component $C$ such that $v(C)\leq 3$. 
By Theorem~\ref{thm_main_2}, $v(C)=2$. 
Therefore $|B^C_q|=|D^C_q|=1$ (see Notation~\ref{nota_prop_component}). 
In this case, $L_C$ denotes the unique linear order defined on $B^C_q$, and 
$\varphi_C $ denotes the unique function from $B^C_q$ to $D^C_q$.  
\end{nota}

The next fact follows from Theorems~\ref{thm_main_1} and \ref{thm1_half_graph}. 
We omit its proof. 

\begin{fac}\label{fac_description_modules}
Given a 2-structure $\sigma$, consider $X\subsetneq V(\sigma)$ such that 
$\sigma[X]$ is prime. 
Suppose that Statement (S5) holds, and $\sigma$ is $\overline{X}\!$-critical. 
Consider a component $C$ of $\Gamma_{(\sigma,\overline{X})}$. 
Let $x\in B_q^C$. 
We have $\sigma-\{x,\varphi_C(x)\}$ is prime. 
Set $$Y=V(\sigma)\setminus\{x,\varphi_C(x)\}.$$
Then, one of the following assertions holds 
\begin{itemize}
\item $\varphi_C(x)\in\langle Y\rangle_\sigma$, $C-x$ is disconnected, $Y$ is the unique nontrivial module of $\sigma-x$, and $x$ is the smallest element of $L_C$;
\item $\varphi_C(x)\in Y_\sigma(\alpha)$, where $\alpha\in X$, $C-x$ is disconnected, $\{\alpha,\varphi_C(x)\}$ is the unique nontrivial module of $\sigma-x$, and $x$ is the smallest element of $L_C$;
\item $\varphi_C(x)\in Y_\sigma(\varphi_C(x^-))$, where $x^-$ is the predecessor of $x$ in $L_C$, $C-x$ is connected, and $\{\varphi_C(x^-),\varphi_C(x)\}$ is the unique nontrivial module of $\sigma-x$. 
\end{itemize}
\end{fac}

\begin{defi}\label{defi_primality}
Given a prime 2-structure $\sigma$, Theorem~\ref{thm_ST} leads Ille~\cite{I93} to introduce the primality graph $\mathbb{P}(\sigma)$ of $\sigma$ as follows. 
It is defined on $V(\sigma)$ as well, and its edges are exactly the non-critical unordered pairs of 
$\sigma$ (see Definition~\ref{defi_critical}). 
Hence, by Theorem~\ref{thm_ST}, $\mathbb{P}(\sigma)$ is nonempty when 
$v(\sigma)\geq 7$. 
The primality graph is an efficient tool to recognize primality in different contexts 
(see \cite{I93} and \cite{BCI14}). 
\end{defi}

Given a 2-structure $\sigma$, consider $X\subsetneq V(\sigma)$ such that $\sigma[X]$ is prime. 
Suppose that Statement (S5) holds, and $\sigma$ is $\overline{X}\!$-critical. 
Note that an element of $\overline{X}$ is not isolated in $\mathbb{P}(\sigma)$ by 
Theorem~\ref{thm_main_4}. 

We end the section by determining the primality graph of a partially critical 2-structure outside the  prime 2-substructure. 
We use the following lemma due to Ille~\cite{I93}. 

\begin{lem}\label{lem_primality_graph}
Consider a prime 2-structure $\sigma$ such that $v(\sigma)\geq 5$. 
Given a critical vertex $v$ of $\sigma$ (see Definition~\ref{defi_critical}), the following three assertions hold
\begin{enumerate}
\item $d_{\mathbb{P}(\sigma)}(v)\leq 2$;
\item if $d_{\mathbb{P}(\sigma)}(v)=1$, then 
$V(\sigma)\setminus(\{v\}\cup N_{\mathbb{P}(\sigma)}(v))$ is a module of 
$\sigma-v$;
\item if $d_{\mathbb{P}(\sigma)}(v)=2$, then 
$N_{\mathbb{P}(\sigma)}(v)$ is a module of $\sigma-v$. 
\end{enumerate}
\end{lem}

The next fact follows from Fact~\ref{fac_description_modules} and 
Lemma~\ref{lem_primality_graph}. 
We omit its proof. 

\begin{fac}\label{fac_description_primality}
Given a 2-structure $\sigma$, consider $X\subsetneq V(\sigma)$ such that 
$\sigma[X]$ is prime. 
Suppose that Statement (S5) holds, and $\sigma$ is $\overline{X}\!$-critical. 
Consider a component $C$ of $\Gamma_{(\sigma,\overline{X})}$ such that 
$v(C)\geq 6$. 
Then, we have 
\begin{equation}\label{E0_fac_description_primality}
\mathbb{P}(\sigma)[V(C)]=\mathbb{P}(C).
\end{equation}
Moreover, the following two assertions hold. 
\begin{enumerate}
\item For each $x\in B_q^C$, if $\varphi_C(x)\in Y_\sigma(\alpha)$, where 
$Y=V(\sigma)\setminus\{x,\varphi_C(x)\}$ and $\alpha\in X$, then 
$N_{\mathbb{P}(C)}(x)=\{\varphi_C(x)\}$ and 
$N_{\mathbb{P}(\sigma)}(x)=\{\alpha,\varphi_C(x)\}$. 
\item For each $x\in B_q^C$, $N_{\mathbb{P}(C)}(x)\neq N_{\mathbb{P}(\sigma)}(x)$ if and only if $\varphi_C(x)\in Y_\sigma(\alpha)$, where 
$Y=V(\sigma)\setminus\{x,\varphi_C(x)\}$ and $\alpha\in X$. 
\end{enumerate}
\end{fac}

\section{A new proof of Theorem~\ref{thm_pi}}\label{A_thm_pi}

\begin{proof}[Proof of Theorem~\ref{thm_pi}]
Let $\sigma$ be a prime 2-structure. 
Consider $X\subsetneq V(\sigma)$ such that $\sigma[X]$ is prime. 
Suppose that $\overline{X}$ is finite and $|\overline{X}|\geq 6$. 

For a contradiction, suppose that for each proper subset $Y$ of $\overline{X}$, we have 
\begin{equation}\label{E1_thm_pi}
\text{ if $\sigma[X\cup Y]$ is prime, then $|\overline{X\cup Y}|$ is odd. }
\end{equation}
For $Y=\emptyset$ in \eqref{E1_thm_pi}, we obtain $|\overline{X}|$ is odd. 
Hence $|\overline{X}|\geq 7$. 
For $Y\subsetneq\overline{X}$, with $|Y|=1,3$ or $5$, it follows from \eqref{E1_thm_pi} that 
$\sigma[X\cup Y]$ is not prime. 
Consequently Statement (S5) holds. 
Since $|\overline{X}|$ is odd, there exists 
$C\in\mathcal{C}(\Gamma_{(\sigma,\overline{X})})$ such that $v(C)$ is odd. 
Since $\sigma$ is prime, it follows from Theorem~\ref{thm_main_1} that 
$\sigma[X\cup V(C)]$ is prime. 
We have $\overline{X}=V(C)\cup\overline{X\cup V(C)}$. 
Since $|\overline{X}|$ and $v(C)$ are odd, we obtain that $|\overline{X\cup V(C)}|$ is even. 
It follows from \eqref{E1_thm_pi} that $V(C)=\overline{X}$. 
Thus $\mathcal{C}(\Gamma_{(\sigma,\overline{X})})=\{\Gamma_{(\sigma,\overline{X})}\}$. 
Since $\sigma$ is prime, it follows from Theorem~\ref{thm_main_1} that 
$\Gamma_{(\sigma,\overline{X})}$ is prime. 
By Proposition~\ref{prop_component}, $\Gamma_{(\sigma,\overline{X})}$ is bipartite. 
Futhermore, $P_5\not\leq \Gamma_{(\sigma,\overline{X})}$ by Lemma~\ref{lem1_component}. 
Therefore, it follows from Proposition~\ref{prop1_half_graph} that 
$\Gamma_{(\sigma,\overline{X})}$ is a half graph, which is impossible because 
$v(\Gamma_{(\sigma,\overline{X})})=|\overline{X}|$ and $|\overline{X}|$ is odd. 

Consequently \eqref{E1_thm_pi} does not hold. 
Therefore, 
there exists $Y\subsetneq\overline{X}$ such that $\sigma[X\cup Y]$ is prime, and $|\overline{X\cup Y}|$ is even. 
Recall that $\overline{X}$ is finite, so $\overline{X\cup Y}$ is as well. 
Hence, by applying several times Theorem~\ref{tEHR} from $\sigma[X\cup Y]$, we obtain distinct $v,w\in\overline{X\cup Y}$ such that $\sigma-\{v,w\}$ is prime. 
\end{proof}

As announced in Subsection~\ref{sub_prime}, we extend Theorem~\ref{thm_mys} as follows. 

\begin{thm}\label{thm_mys_ext}
Given a prime 2-structure $\sigma$, consider $X\subsetneq V(\sigma)$ such that 
$\sigma[X]$ is prime. 
Suppose that $$\text{$q_{(\sigma,\overline{X})}^a\neq\emptyset$ (see Notation~\ref{nota_a+s})}.$$
If $\overline{X}$ is finite and $|\overline{X}|\geq 4$, then there exist distinct $v,w\in\overline{X}$ such that $\sigma-\{v,w\}$ is prime. 
\end{thm}

\begin{proof}
By Theorem~\ref{thm_pi}, we can assume that $|\overline{X}|=4$ or $5$. 
If $|\overline{X}|=4$, then it suffices to apply Theorem~\ref{tEHR}. 
Hence suppose that $|\overline{X}|=5$. 
For a contradiction, suppose that Statement (S3) holds. 
It follows from Theorem~\ref{thm_main_1} that for each component $C$ of 
$\Gamma_{(\sigma,\overline{X})}$, we have $v(C)=2$ or $v(C)\geq 4$ and $C$ is prime. 
Since $|\overline{X}|=5$, we obtain that $\Gamma_{(\sigma,\overline{X})}$ is connected. 
Thus $\Gamma_{(\sigma,\overline{X})}$ is prime. 
Since $\Gamma_{(\sigma,\overline{X})}$ is connected, it follows from the first assertion of Proposition~\ref{prop_component} that 
$p_{(\sigma,\overline{X})}=q_{(\sigma,\overline{X})}$, and 
$q_{(\sigma,\overline{X})}$ has two elements, denoted by $B_q$ and $D_q$. 
Moreover, $\Gamma_{(\sigma,\overline{X})}$ is bipartite, with bipartition 
$\{B_q,D_q\}$. 
Since $\Gamma_{(\sigma,\overline{X})}$ is prime and bipartite, we have 
$\Gamma_{(\sigma,\overline{X})}\simeq P_5$. 
Hence $K_2\oplus K_2\leq\Gamma_{(\sigma,\overline{X})}$. 
Thus, there exists distinct $x,x'\in B_q$ and distinct $y,y'\in D_q$ such that 
$\{x,y\},\{x',y'\}\in E(\Gamma_{(\sigma,\overline{X})})$ and 
$\{x,y'\},\{x',y\}\not\in E(\Gamma_{(\sigma,\overline{X})})$. 
It follows from Fact~\ref{fac3_first_results} that 
$B_q,D_q\in q_{(\sigma,\overline{X})}^s$, which contradicts 
$q_{(\sigma,\overline{X})}^a\neq\emptyset$. 
Consequently, Statement (S3) does not hold. 
Hence, there exists $Y\subseteq\overline{X}$ such that $|Y|=3$ and 
$\sigma[X\cup Y]$ is prime, which completes the proof because 
$|\overline{X}|=5$. 
\end{proof}


\begin{thebibliography}{99}
\bibitem{BBH15}H. Belkhechine, I. Boudabbous, K. Hzami, The prime tournaments 
$T$ with $|W_5(T)|=|T|-2$, Turkish J. Math. 39 (2015) 570--582. 
\bibitem{B94}P. Bonizzoni, Primitive 2-structures with the $(n-2)$-property,
Theoret. Comput. Sci. 132 (1994), 151--178.
\bibitem{BDY18}I. Boudabbous, J. Dammak, M. Yaich, 
Prime criticality and prime covering, to appear in 
Math. Rep. (Bucur.). 
\bibitem{BI07}I. Boudabbous, P. Ille, Critical and infinite directed graphs, 
Discrete Math. 307 (2007) 2415--2428. 
\bibitem{BI09}Y. Boudabbous, P. Ille, 
Indecomposability graph and critical vertices of an indecomposable graph, 
Discrete Math. 309 (2009) 2839--2846. 
\bibitem{BCI14}A. Boussa\"{\i}ri, A. Cha\"{\i}cha\^a, P. Ille, 
Indecomposability graph and indecomposability recognition, 
Proceedings of ROGICS'08, European J. Combin. 37 (2014) 32--42. 
\bibitem{BDI08} A. Breiner, J. Deogun, P. Ille, 
Partially critical indecomposable graphs, 
Contrib. Discrete Math. 3 (2008) 40--59. 
\bibitem{EHR99}A. Ehrenfeucht, T. Harju, G. Rozenberg, 
The Theory of 2-Structures, A Framework for Decomposition and 
Transformation of Graphs, 
World Scientific, Singapore, 1999. 
\bibitem{EH85}P. Erd\H{o}s, A. Hajnal, 
Chromatic number of finite and infinite graphs and hypergraphs, 
Discrete Math. 53 (1985) 281--285. 
\bibitem{I93}P. Ille, Recognition problem in reconstruction 
for decomposable  relations, in B. Sands, N. Sauer, R. Woodrow (Eds.), 
Finite and Infinite Combinatorics in Sets and Logic, 
Kluwer Academic Publishers, 1993, 189--198.
\bibitem{I94}P. Ille, Graphes ind\'ecomposables infinis, 
C.R. Acad. Sci. Paris S\'erie I Math. 318 (1994) 499--503.
\bibitem{I97}P. Ille, Indecomposable graphs, 
Discrete Math 173 (1997) 71--78. 
\bibitem{I05}P. Ille, La d\'ecomposition intervallaire des structures binaires, 
La Gazette des Math\'ematiciens 104 (2005), 39--58. 
\bibitem{I05b}P. Ille, A characterization of the indecomposable and infinite graphs, 
Proceedings of the 13th Symposium of the Tunisian Mathematical Society 
(Sousse, Tunisia, 2005), Glob. J. Pure Appl. Math. 1 (2005) 272--285.
\bibitem{IR14}P. Ille, R. Villemaire, Recognition of prime graphs from a prime subgraph, Discrete Math. 327 (2014) 76--90. 
\bibitem{R82}J.G. Rosenstein, Linear orderings, Academic Press, New York, 1982.
\bibitem{S11}M.Y. Sayar, 
Partially critical tournaments and partially critical supports, 
Contrib. Discrete Math. 6 (2011) 52--76. 
\bibitem{S92}J. Spinrad, P4-trees and substitution decomposition, 
Discrete Appl. Math. 39 (1992) 263--291. 
\bibitem{ST93}J.H. Schmerl, W.T. Trotter, Critically indecomposable 
partially ordered sets, graphs, tournaments and other binary relational structures, 
Discrete Math 113 (1993), 191--205.
\end{thebibliography}
\end{document}